\documentclass[reqno,11pt]{amsart}
\usepackage{amssymb,amsthm,amsfonts,amstext}
\usepackage{amsmath}
\usepackage{dsfont}
\usepackage{helvet}        
\usepackage{url}
\usepackage{bbm}
\usepackage{nicefrac}
\makeatletter \@addtoreset{equation}{section} \makeatother

\renewcommand\thefigure{\thesection.\@arabic\c@figure}
\renewcommand\thetable{\thesection.\@arabic\c@table}
\newtheorem{theorem}{Theorem}[section]
\newtheorem{lemma}[theorem]{Lemma}
\newtheorem{proposition}[theorem]{Proposition}
\newtheorem{corollary}[theorem]{Corollary}

\theoremstyle{definition}
\newtheorem{remark}[theorem]{Remark}

\newcommand{\mc}[1]{{\mathcal #1}}

\newcommand{\bb}[1]{{\mathbb #1}}
\renewcommand{\Re}{\operatorname{Re}}
\renewcommand{\Im}{\operatorname{Im}}

\newcommand{\<}{\langle}
\renewcommand{\>}{\rangle}

\DeclareMathOperator*\varlim{\vphantom{p}lim}

\newcommand{\N}{\mathcal{N}}

\newcommand{\ud}{\mathrm{d}}
\usepackage{color}

\title{{Thermalisation for  Small Random Perturbations of Dynamical Systems}}
\author{Barrera, Gerardo}
\address{University of Alberta, Department of Mathematical and Statistical Sciences. Central Academic Building. $116$ Street and $85$ Avenue. Postal Code: T$6$G--$2$G$1$. Edmonton, Alberta, Canada.}
\email{barrerav@ualberta.ca }
\author{Jara, Milton}
\address{Instituto de Matem\'atica Pura e Aplicada, IMPA. Estrada Dona Castorina $110$, Jardim Bot\^anico. Postal Code: $22460$--$320$.
Rio de Janeiro, Rio de Janeiro, Brazil.}
\email{monets@impa.br}

\begin{document}
\maketitle
\begin{abstract}
We consider an ordinary differential equation  with a unique hyperbolic attractor at the origin, to which we add a small random perturbation.  It is known that under general conditions, the solution of this stochastic differential equation converges exponentially fast to an equilibrium distribution.
We show that the convergence occurs abruptly: in a time window of small size compared to the natural time scale of the process, the distance to equilibrium drops from its maximal possible value to near zero, and only after this time window the convergence is exponentially fast. This is what is known as the cut-off phenomenon in the context of Markov chains of increasing complexity. In addition,
we are able to give general conditions to decide whether the distance to equilibrium converges in this time window to a universal function, a fact known as profile cut-off.
\end{abstract}

\markboth{Thermalisation for Perturbations of Dynamical Systems}{Thermalisation for Perturbations of Dynamical Systems}
\section{Introduction}
Our main goal is the study of the convergence to equilibrium for a family of stochastic small random perturbations of a given dynamical system in $\mathbb{R}^d$. 
Consider an ordinary differential equation with a unique hyperbolic global attractor. 
Without loss of generality, we assume that the global attractor is located at the origin. 
Under general conditions, as time goes to infinity, any solution of this differential equation approaches the origin exponentially fast. 
We perturb the deterministic dynamics by a Brownian motion of small intensity.
It is well known that, again under very general conditions, as time goes to infinity, any solution of this stochastic differential equation converges in distribution to an equilibrium law.
The convergence can be improved to hod with respect to the total variation distance.
The theory of Lyapunov functions allows to show that this convergence, for each fixed perturbation, is again exponentially fast. We show that the convergence occurs  abruptly: 
when the intensity of the noise goes to zero, the total variation distance between the law of the stochastic dynamics and the law of its equilibrium in a time window around the cut-off time decreases from one to near zero abruptly, 
and only after this time window the convergence is exponentially fast. 
This fact is known as cut-off phenomenon.
Moreover,
when a properly normalised $\omega$-limit set of the initial datum of the deterministic differential equation is contained in a sphere,
we are able to prove convergence of the distance to equilibrium to a universal function, a fact known as profile cut-off or profile thermalisation in the context of ergodic Markov processes.

To be more precise, we are concerned about the abrupt convergence to equilibrium in the total variation distance for systems of the form:
\begin{equation}\label{over}
\left\{
\begin{array}{r@{\;=\;}l}
\ud x^{\epsilon}(t)&-F(x^{\epsilon}(t))\ud t + \sqrt{\epsilon}\ud B(t)\quad  
\textrm{  for } t\geq 0,\\
x^{\epsilon}(0) & x_0.
\end{array}
\right.
\end{equation}
where $F$ is a given vector field with a unique hyperbolic fixed point and $\{B(t):t\geq 0\}$ is a standard Brownian motion.
Notice that systems described by the stochastic differential equation \eqref{over} are not necessarily reversible.
In statistical physics, equation \eqref{over} is known as an {\em overdamped Langevin dynamics}, and it is used to model fluctuations of stationary states.

In the small noise asymptotics, the stochastic dynamics \eqref{over} fluctuates around the attractor of the deterministic dynamics which is called {\em{relaxation dynamics}} or {\em{zero-noise dynamics}}. Assuming that the deterministic dynamics is 
{\em{strongly coercive}} together with some 
{\em{growth condition}} on $F$,
when the intensity $\epsilon$ of the noise goes to zero, in a time windows of small size compared to the natural time scale  of the process, the total variation distance to equilibrium drops 
 from near one to near zero.

Dynamical systems subjected to small Gaussian perturbations have been studied extensively,
see the book of M.~Freidlin \& A.~Wentzell \cite{FW} which discusses this problem in
great detail; see also  M.~Freidlin \& A.~Wentzell \cite{FW1}, \cite{FW2}, M.~Day \cite{MD}, \cite{MD1} and W.~Siegert \cite{SI}.
This treatment has inspired many works and
considerable effort was concerned about purely local phenomena, {\em{i.e.}},
on the computation of exit times and exit probabilities from neighbourhoods of fixed
points that are carefully stipulated not to contain any other fixed point of
the deterministic dynamics.

The theory of large deviations allows to solve the exit problem from the domain of attraction of a stable point.
It turns out that the mean exit time is exponentially large in the small noise parameter, and its logarithmic rate is proportional to the height of the potential barrier that the trajectories have to overcome.
Consequently, for a multi-well potential one can obtain a series of exponentially non-equivalent time scales given by the wells-mean exit times. Moreover, the normalised exit times are asymptotically exponentially distributed and have a memoryless property, for further details see A.~Galves, E.~Olivieri \& M.~Vares \cite{GOV}, E.~Olivieri \& M.~Vares \cite{OV} and C.~Kipnis \& C.~Newman \cite{KN}.
There are situations in which the analysis at the level of large deviations is not enough,
and it is necessary the study of distributional scaling limits for the exit distributions, for more details see Y.~Bakhtin \cite{YB1}
and \cite{YB2}.

The cut-off phenomenon was extensively studied 
in the eighties to describe the phenomenon of abrupt convergence that appears in models of card shuffling, Ehrenfest urns and random transpositions, see for instance D.~Aldous \& P.~Diaconis  \cite{AD} and \cite{AD1}. In general, it is a challenging problem to prove  that an specific family of stochastic models exhibit or does not exhibit a  cut-off phenomenon. It requires a complete understanding of the dynamics of the specific random process. 

Since the appearance of \cite{AD} many families of stochastic processes have been shown to have similar properties.
Various notions of cut-off have been proposed; see
J.~Barrera \& B.~Ycart
\cite{BY} and
P.~Diaconis \cite{DI} for an account. 
We refer to the book of D.~Levin et al. (\cite{LPW}, Chapter $18$) for an introduction of the subject in the Markov chain setting, L.~Saloff-Coste \cite{SA} provides an extensive list of random walks for which the cut-off phenomenon holds,  P.~Diaconis \cite{DI} for a review on the finite Markov chain case, S.~Mart\'inez and B.~Ycart \cite{MY} for the case of Markov chains with countably infinite state space,  G.~Chen and L.~Saloff-Coste \cite{CSC} for Brownian motions on a compact Riemann manifold, B.~Lachaud \cite{BL} and G.~Barrera \cite{BA} for Ornstein-Uhlenbeck processes on the line and G.~Barrera and M.~Jara \cite{BJ} for stochastic small perturbations of one-dimensional dynamical systems.

Roughly speaking, thermalisation or window cut-off holds for a family of stochastic systems, when convergence to equilibrium happens in a time window which is small compared to the total running time of the system.
Before a certain ``cut-off time'' those processes stay far from equilibrium with respect to some suitable distance; in a time window of smaller order the processes get close to equilibrium, and after a time window that convergence to equilibrium happens exponentially fast.

Alternative names are threshold phenomenon and abrupt convergence.
When the distance to equilibrium at the time window can be well approximated by some profile function, we speak about profile cut-off.
Sequences of stochastic processes for which an explicit profile cut-off can be determined are scarce. Explicit profiles are usually out of reach, in particular for the total variation distance.
In general, the existence of the phenomenon is proven through a precise
estimation of the sequence of cut-off times and this precision comes at a high technical
price, for more details see J.~Barrera, O.~Bertoncini \& R.~Fern\'andez \cite{BBF}.

The main result of this article, Theorem \ref{main}, states that when the deterministic dynamics is
{\em{strongly coercive}} and satisfies some {\em{growth condition}}, the family of perturbed dynamics presents a thermalisation (windows cut-off) as we describe in Section \ref{notres}.
Moreover, in Corollary \ref{pt} and Corollary \ref{dpt} we give a necessary and sufficient condition for having profile thermalisation (profile cut-off).
We point out that our condition is always satisfied by reversible dynamics; {\em{i.e.}}, when $F(x)=\nabla V(x)$, $x\in \mathbb{R}^d$, and also for a large class of dynamics that are non-reversible.
Non-reversible dynamics naturally appear for example in polymeric fluid dynamics or Wigner-Fokker-Planck equations, see
A.~Arnold, J.~Carrillo \& C.~Manzini \cite{ACM} and B.~Jourdain, C.~Le Bris, T.~Leli\`evre \& F.~Otto \cite{JBLO}.
Non-reversible systems arise in the theory of activated process in glasses and other disordered materials, chemical reactions far from equilibrium, stochastic modelled computer networks, evolutionary biology and  theoretical ecology, see R.~Maier \& D.~Stein \cite{MS1} and \cite{MS2}.

Notice that the set of symmetric matrices is not open. In particular, reversibility is not a generic property of dynamical systems. On the other hand, hyperbolicity is an open property, meaning that it is stable under small perturbations of the vector field.
Moreover, in general for the non-reversible case, there is not an explicit formula for the invariant measure of the random dynamics \eqref{over} as in the reversible case. 
For reversible dynamics, analytic methods from quantum mechanics have been used to compute asymptotic expansions in the diffusivity $\sqrt{\epsilon}$. The strong point is that full asymptotic expansions in $\sqrt{\epsilon}$ and sharp estimates can be done. However, so far only applicable for reversible diffusion process. For more details, see \cite{BEGK} and \cite{BEGK1}.
Therefore, it is desirable to have a treatment that does not rely on these properties, namely reversibility and/or explicit knowledge of invariant measures.

Our idea is to carry out this asymptotic expansion in $\sqrt \epsilon$ by probabilistic methods. It turns out that the hyperbolicity of the underlying dynamics can be used to show that a second-order expansion gives a description of the original dynamics which is good for times much larger than the time at which equilibration occurs. This expansion is the same introduced in \cite{BJ}. Notice that in the one-dimensional case, the stochastic dynamics is always reversible. Therefore, profile cut-off always holds and a more refined analysis of this expansion is needed in order to be able to discern whether profile cut-off holds or not. A consequence of our analysis is an $L^1$-version of the local central limit theorem for the invariant measure of \eqref{over}, which could be of independent interest.

This material is organized as follows. Section \ref{notres} describes the model and states the main result besides establishing the basic notation and definitions. Section \ref{mulana} provides sharp estimates on the asymptotics of related linear approximations which are the main ingredient in order to prove the main result in the end of this section. 
Finally, we provide an Appendix which is divided in three sections as follows:
Section \ref{pgdis} gives useful properties for the total variation distances between Gaussian distributions. Section
\ref{B.tomate} and \ref{apend3} provide the rigorous arguments about the deterministic dynamics  and
the stochastic dynamics,   respectively,
 that we omit in Section \ref{mulana} to make the presentation more fluid.

\section{Notation and results}\label{notres}
In this section we rigorously state the family of stochastically perturbed dynamical systems that we are considering and the results we prove.

\subsection{The dynamical system}
Let $F:\bb R^d \to \bb R^d$ be a vector field of class $\mc C^2(\bb R^d,\bb R^d)$. For each $x \in \bb R^d$, let $\{\varphi(t,x): t\in [0,\tau_x)\}$ be the solution of the deterministic differential equation:
\begin{equation}
\label{EDO}
\left\{
\begin{array}{r@{\;=\;}l}
\tfrac{\ud}{\ud t}\varphi(t) & - F(\varphi(t))\quad \textrm{ for }
 0\leq t <\tau_x,\\
\varphi(0) & x
\end{array}
\right.
\end{equation}
where $\tau_x$ denotes the {\em explosion time}.
Since $F$ is smooth, this equation has a unique solution.
Since we have not imposed any growth condition on $F$, $\tau_x$ may be finite.
We denote by $\|\cdot\|$ the Euclidean norm in $\mathbb{R}^d$ and by $\<\cdot,\cdot\>$ the standard inner product of $\mathbb{R}^d$.
Under the condition
\[
\sup\limits_{z\in \bb R^d}{\frac{\<z,-F(z)\>}{1+\|z\|^2}}<+\infty,
\]
an straightforward application of the
Lemma \ref{gin} (Gronwall's inequality) 
 implies that
the explosion time $\tau_x$ is infinite for any $x\in \bb R^d$. Later on, we will make stronger assumptions on $F$, so we will assume that the explosion time is always infinite without further comments.

We call the family $\{\varphi(t,x): t \geq 0, x \in \bb R^d\}$ the {\em dynamical system} associated to $F$.
We say that a point $y\in \bb R^d$ is a {\em fixed point} of \eqref{EDO} if $F(y)=0$.
In that case $\varphi(t,y)=y$ for any $t\geq 0$.

Let $y$ be a fixed point of \eqref{EDO}. We say that $x \in \bb R^d$ belongs in the {\em basin of attraction} of $y$ if
\[
\lim_{t \to +\infty} \varphi(t,x) =y.
\]
We say that $y$ is an {\em attractor} of \eqref{EDO} if the set
\[
U_y = \{ x \in \bb R^d: x \text{ is in the basin of attraction of } y\}
\]
contains an open ball centered at $y$. If $U_y = \bb R^d$ we say that $y$ is a {\em global attractor} of \eqref{EDO}. We say that $y$ is a {\em hyperbolic} fixed point of
\eqref{EDO} if $\mathrm{Re}(\lambda) \neq 0$ for any eigenvalue $\lambda$ of the Jacobian matrix $DF(y)$. By the Hartman-Grobman Theorem (see Theorem (Hartman) page $127$ of \cite{PE} or the celebrated paper of P.~Hartman \cite{HA}), a hyperbolic fixed point $y$ of \eqref{EDO} is an attractor if and only if $\mathrm{Re}( \lambda) >0$ for any eigenvalue $\lambda$ of the matrix $DF(y)$. From now on, we will always assume that
\[
0   \text{ is a hyperbolic attractor of }   \eqref{EDO}.
\]
In that case, for any $x \in U_0$ the asymptotic behaviour of $\varphi(t,x)$ as $t \to +\infty$ can be described in a very precise way.

A sufficient condition for $0$ to be a global attractor of \eqref{EDO} is the following {\em coercivity} condition: there exists a positive constant $\delta$ such that
\begin{equation}
\tag{C}
\label{C1}
\<x,F(x)\> \geq \delta \|x\|^2\quad \text{ for any } x \in \bb R^d.
\end{equation}
Notice that
\[
\frac{\ud }{\ud t}\|\varphi(t)\|^2=2\<\varphi(t),\frac{\ud }{\ud t}\varphi(t)\>=\<\varphi(t),-F(\varphi(t))\>\leq -2\delta \|\varphi(t)\|^2
\]
for any $t\geq 0$. Then Lemma \ref{gin} allows us to deduce that
\begin{equation}\label{mono}
\|\varphi(t,x)\| \leq \|x\| e^{-\delta t}\quad \text{ for any } x \in \bb R^d \text{ and any } t \geq 0.
\end{equation}
In other words, $\varphi(t,x)$ converges to $0$ exponentially fast as $t\rightarrow +\infty$. 
Notice that the eigenvalues of the Jacobian matrix of $F$ at zero, $DF(0)$, might be complex numbers. 
Recall that for any $\lambda\in \mathbb{C}$ and 
$v\in \mathbb{C}^d$, 
\begin{align*}
\lambda v&=(\Re(\lambda)+i\Im(\lambda))(\Re(v)+i\Im(v))\\
&=\Re(\lambda)\Re(v)-\Im(\lambda)\Im(v)+
i\left(\Im(\lambda)\Re(v)+\Re(\lambda)\Im(v)\right).
\end{align*}
From \eqref{C1}
we have $\mathrm{Re}(\lambda) \geq \delta$ for any eigenvalue $\lambda$ of $DF(0)$.
Let $v\in\mathbb{C}^d$ an eigenvector associated to the eigenvalue $\lambda$ of $DF(0)$. Then
\begin{equation}
-(\mathrm{Re}(\lambda)-\delta)\|\mathrm{Im}(v)\|^2\leq \mathrm{Im}(\lambda)\<\mathrm{Re}(v),\mathrm{Im}(v)\>\leq (\mathrm{Re}(\lambda)-\delta)\|\mathrm{Re}(v)\|^2.
\label{control}
\end{equation}
Particularly, from \eqref{control} we have that  \eqref{C1} does not allow to control the imaginary part of the eigenvalues of $DF(0)$. Typically and roughly speaking, the dynamical system associated to \eqref{EDO} is an ``uniformly contracting spiral".

The following Lemma provides us the asymptotics of $\varphi(t)$ as $t$ goes to $+\infty$. It will be important for determining the cut-off time and time window.
\begin{lemma}
\label{asymp}
Assume that \eqref{C1} holds. Then
for any $x_0\in \mathbb{R}^{d}\setminus\{0\}$  there exist $\lambda:=\lambda(x_0)>0$, $\ell:=\ell(x_0)$, $m:=m(x_0) \in \bb \{1,\ldots,d\}$, $\theta_1:=\theta_1(x_0),\dots,\theta_m:=\theta_m(x_0) \in [0,2\pi)$, $v_1:=v_1(x_0),\dots,v_m:=v_m(x_0)$ in $\bb C^d$ linearly independent and $\tau:=\tau(x_0)>0$ such that
\[
\lim_{t \to +\infty} \Big\| \frac{e^{\lambda t}}{t^{\ell-1}} \varphi(t+\tau,x_0) - \sum\limits_{k=1}^m e^{i\theta_k t} v_k \Big\|=0.
\]
\end{lemma}

This lemma will be proved in Appendix \ref{B.tomate}, where we give  more detailed description of the constants and vectors appearing in this lemma. We can anticipate that the numbers $\lambda\pm i\theta_k$, $k=1,\ldots,m$ are eigenvalues of $DF(0)$ and that the vectors $v_k\in \mathbb{C}^{d}$, $k=1,\ldots,m$ are elements of the Jordan decomposition of the matrix $DF(0)$.

\subsection{The cut-off phenomenon}\label{section2.2.}
Let $\mu$, $\nu$ be two probability measures in $(\bb R^d,\mc{B}(\bb R^d))$. We say that a probability measure $\pi$ in
$(\bb R^d \times \bb R^d,\mc{B}(\bb R^d\times \bb R^d))$ is a {\em coupling} between $\mu$ and $\nu$ if for any Borel set $B\in \mc{B}{(\bb R^d)}$,
\[
\pi(B \times \bb R^d) = \mu(B) \text{ and } \pi(\bb R^d \times B) = \nu(B).
\]
In that case we say  that $\pi \in \mc C(\mu,\nu)$. The {\em total variation} distance between $\mu$ and $\nu$ is defined as
\[
\ud_{\mathrm{TV}}(\mu,\nu) = \inf_{\pi \in \mc C(\mu,\nu)} \pi \big\{(x,y) \in \bb R^d \times \bb R^d: x \neq y\big\}.
\]
Notice that the diameter with respect to $\ud_{\mathrm{TV}}(\cdot,\cdot)$ of the set $\mc M_1^+(\bb R^d,\mc{B}(\bb R^d))$ of probability measures defined in $(\bb R^d,\mc{B}(\bb R^d))$ is equal to $1$.
If $X$ and $Y$ are two random variables in $\mathbb{R}^d$ which are defined in the same measurable space $(\Omega,\mathcal{F})$, we write $\ud_{\mathrm{TV}}(X,Y)$ instead of $\ud_{\mathrm{TV}}(\mathbb{P}(X\in \cdot),\mathbb{P}(Y\in \cdot))$.

For simplicity, we also write 
$\ud_{\mathrm{TV}}(X,\mu_{Y})$ in place of
$\ud_{\mathrm{TV}}(X,Y)$, where $\mu_{Y}$ is the distribution of the random variable $Y$.
For an account of the equivalent formulations of the total variation distance (normalised or not
normalised), we recommend the book of A.~Kulik (\cite{KUL}, Chapter $2$).

For any $\epsilon\in (0,1]$, let $x^{\epsilon}$ be the continuous time stochastic process $\{x^{\epsilon}(t):t\geq 0\}$.
We say that a family of stochastic processes $\{x^{\epsilon}\}_{\epsilon \in (0,1]}$ has {\em thermalisation} at position $\{t^\epsilon\}_{\epsilon \in (0,1]}$, window $\{\omega^\epsilon\}_{\epsilon \in (0,1]}$ and state $\{\mu^\epsilon\}_{\epsilon \in (0,1]}$ if
\begin{itemize}
\item[i)]
\[
\lim_{\epsilon \to 0}{t^{\epsilon}}=+\infty
\quad\textrm{ and }\quad \lim_{\epsilon \to 0} \frac{\omega^\epsilon}{t^\epsilon} =0,
\]

\item[ii)]
\[
\varlim_{c \to +\infty} \limsup_{\epsilon \to 0} \ud_{\mathrm{TV}} (x^\epsilon({t^\epsilon+c\omega^\epsilon}), \mu^\epsilon)=0,
\]
\item[iii)]

\[
\lim_{c \to -\infty} \liminf_{\epsilon \to 0} \ud_{\mathrm{TV}} (x^\epsilon({t^\epsilon +c\omega^\epsilon}), \mu^\epsilon)=1.
\]
\end{itemize}

If for any $\epsilon\in (0,1]$, $x^\epsilon$ is a Markov process with a unique invariant measure and $\mu^\epsilon$ is the invariant measure of the process $x^{\epsilon}$ we say that the family $\{x^{\epsilon}\}_{\epsilon \in (0,1]}$ presents {\em thermalisation}  or {\em window cut-off}.

If in addition to i) there is a continuous function $G: \bb R \to [0,1]$ such that $G(-\infty)=1$, $G(+\infty)=0$ and
\[
\text{ii')}\;\; \lim_{\epsilon \to 0} \ud_{\mathrm{TV}} ( x^\epsilon({t^\epsilon+c\omega^\epsilon}), \mu^\epsilon)=:G(c)
\quad \textrm{ for any } c\in \mathbb{R},
\]
we say that there is {\em profile  thermalisation} or {\em profile cut-off}. Notice that $\text{ii')}$ implies $\text{ii)}$ and $\text{iii)}$, and therefore profile thermalisation (respectively profile cut-off) is a stronger notion than thermalisation (respectively window cut-off).

\subsection{The overdamped Langevin dynamics}\label{tsp}
Let $\{B(t): t \geq 0\}$ be a standard Brownian motion in $\bb R^d$ and let $\epsilon \in (0,1]$ be a scaling parameter. Let $x_0 \in U_0 \setminus \{0\}$ and let $\{x^\epsilon(t,x_0): t \geq 0\}$ be the solution of the following stochastic differential equation:
\begin{equation}
\label{SEDO}
\left\{
\begin{array}{r@{\;=\;}l}
\ud x^\epsilon(t) & - F(x^\epsilon(t)) \ud t + \sqrt \epsilon \ud B(t)
\quad \textrm{  for } t\geq 0,\\
x^\epsilon(0) & x_0.
\end{array}
\right.
\end{equation}

Stochastic differential equation \eqref{SEDO} is used in molecular modelling. In that context $\epsilon=2\kappa \tau$,
where $\tau$ is the temperature of the system and $\kappa$ is the Boltzmann constant. 
In statistical physics, equation \eqref{SEDO} has a computational interest to modelling a sample of a Gibbs measure in high-dimensional Euclidean spaces. 
Denote by $(\Omega,\mc F,\bb P)$ the probability space where $\{B(t): t \geq 0\}$ is defined and denote by $\bb E$ the expectation with respect to $\bb P$. Notice that \eqref{SEDO} has a unique strong solution (see Remark $2.1.2$ page $57$ of \cite{SI} or Theorem $10.2.2$ of \cite{SV}), and therefore $\{x^\epsilon(t,x_0): t \geq 0\}$ can be taken as a stochastic process in the same probability space
$(\Omega,\mc F,\bb P)$. 

In order to avoid unnecessary notation, we write $\{x^\epsilon(t): t \geq 0\}$ instead of $\{x^\epsilon(t,x_0): t \geq 0\}$ and $\{\varphi(t): t \geq 0\}$ instead of $\{\varphi(t,x_0): t \geq 0\}$.
Since  $\epsilon \in (0,1]$, for simplicity, we write $\lim\limits_{\epsilon \to 0}$ instead of 
$\lim\limits_{\epsilon \to 0^+}$. 

Our aim is to describe in detail the asymptotic behaviour of the law of $x^\epsilon(t)$ for large times $t$, as $\epsilon \to 0$. In particular, we are interested in the law of $x^\epsilon(t)$ for times $t$ of order $\mc O(\log (1/\epsilon))$, where {\em thermalisation} or {\em window cut-off} phenomenon appears.

Under \eqref{C1}, for any $\epsilon\in (0,1]$,  the process $\{x^\epsilon(t): t \geq 0\}$ is uniquely ergodic with stationary measure $\mu^\epsilon$, see Lemma \ref{ec4} for details. 
 Moreover, the process is strongly Feller.
In particular, the process visits infinitely often every non-empty open set of the state space $\mathbb{R}^d$.
The stationary measure $\mu^{\epsilon}$  is absolutely continuous with respect to the Lebesgue measure in $\bb R^d$. The density $\rho^{\epsilon}$ of $\mu^{\epsilon}$ is smooth and solves the stationary Fokker-Planck equation:
\begin{equation*}
\frac{\epsilon}{2}\sum\limits_{j=1}^{d}\frac{\partial^2 }{\partial x_j^2}(\rho^{\epsilon}(x))+\sum\limits_{j=1}^{d} \frac{\partial }{\partial x_j}(F_j(x)\rho^{\epsilon}(x))=0\quad \textrm{  for any } x\in \mathbb{R}^d,
\end{equation*}
where $F=(F_1,\ldots,F_d)^T$, for details see \cite{SI} (pages $60$-$63$). 
When the process is reversible, {\em{i.e.}}, $F(x)=\nabla V(x)$, $x
\in \mathbb{R}^d$, for some scalar function $V$ (also called potential), the stationary measure $\mu^\epsilon$ is of the Gibbs type:
\begin{equation}
\label{invatt}
\mu^{\epsilon}(\ud x)=\frac{1}{\mathcal{Z}^\epsilon}e^{-\frac{2V(x)}{\epsilon}}\ud x,\quad
\textrm{ where }\quad \mathcal{Z}^\epsilon=\int\limits_{\mathbb{R}^d}e^{-\frac{2V(x)}{\epsilon}}\ud x<+\infty.
\end{equation}
The normalised constant $\mathcal{Z}^\epsilon$ is called the partition function. 
If the vector field $F$ can be decomposed as 
\[F(x)=\nabla V(x)+b(x)\quad \textrm{ for any } x\in \mathbb{R}^d,\]  where $V:\mathbb{R}^d\rightarrow \mathbb{R}$ is a scalar function and $b:\mathbb{R}^d\rightarrow \mathbb{R}^d$ is a vector field which satisfies the
divergence-free condition:
\[\textrm{div}\left(e^{-\frac{2}{\epsilon}V(x)}b(x)\right)=0\quad  \textrm{ for any }  x\in \mathbb{R}^d,\]
then under some appropriate assumptions on $V$ at infinity, {\em{i.e.}}, 
\[
\frac{1}{2}\|\nabla V(x)\|^2-\Delta V(x)\rightarrow +\infty \quad \textrm{ as } \|x\|\rightarrow +\infty,
\]  
the probability measure $\mu^{\epsilon}$ given by \eqref{invatt} remains stationary for \eqref{SEDO}. For details see  \cite{HHS}, \cite{VI1}, \cite{JA} and \cite{LNP}. In this situation, using the Laplace Method, asymptotics as $\epsilon\to 0$ for $\mu^{\epsilon}$ can be obtained, see \cite{HWANG} and \cite{AHWANG} for further details.
In general, the equilibrium measure can be expressed as an integral of a Green function, but aside from a few simple cases, there are no closed expressions for it.
In this case, the Freidlin-Wentzell theory implies that the non-Gibbs measure $\mu^{\epsilon}$ is equivalent to a Gibbs measure with a ``quasi-potential" $\tilde{V}$ playing the role of the potential energy, see for instance \cite{MS4}, \cite{MS1}, \cite{MS3} and \cite{MS5}. However, the study of the regularity of the quasi-potential is a non-trivial mathematical issue, for details see \cite{BK}. For our purposes, no transverse condition on the vector field $F$ is assumed and also we do not need that the Gibbs measure remains stationary  for \eqref{SEDO}, for further details  see \cite{BR} and the references therein.

In many theoretical or applied problems involving ergodic processes, it
is important to estimate the time until the distribution of the process is close to its equilibrium distribution.
Under some {\it{strong coercivity condition}} and {\it{growth condition}} that we will state precisely in Section \ref{resultados},
we will prove that the law of $x^\epsilon(t)$ converges in total variation distance to $\mu^\epsilon$ in a {\em time window}
\begin{equation}
\label{win}
w_{}^\epsilon:= \frac{1}{\lambda}+o(1)
\end{equation}
of order $\mc O(1)$ around the {\em{cut-off time}} 
\begin{equation}
\label{mix}
t_{\mathrm{mix}}^\epsilon:= \frac{1}{2\lambda} \ln \left(1/\epsilon\right) +
\frac{\ell-1}{\lambda} \ln \left(\ln \left(1/\epsilon\right)\right)+\tau,
\end{equation}
where $\lambda$, $\ell$ and $\tau$ are the positive constants associated to $x_0$ in Lemma \ref{asymp}.

If we only assume that $0$ is a hyperbolic attractor of \eqref{EDO}, we can not rule out the existence of other attractors. These attractors are accessible to the stochastic dynamics \eqref{SEDO} (a large part of the celebrated book of M. Freidlin \& A.~Wentzell \cite{FW} is devoted to the study of this problem).
However, in this situation the other attractors are not accessible until times of order $\mc O(e^{\nicefrac{c}{\epsilon}})$, where $c$ is a positive constant, and we prove that the law of $x^\epsilon(t)$ converges to a Gaussian random variable on a time window of order $\mc O(1)$ around $t_{\mathrm{mix}}^\epsilon$.

The exact way on which this convergence takes place is the content of the following section.

\subsection{Results}\label{resultados}
Denote by  $\mathcal{G}{\left(v,\Xi \right)}$ the Gaussian distribution in $\mathbb{R}^d$ with vector mean $v$ and positive definite covariance matrix $\Xi $.
Let $\mathrm{I}_d$ be the identity $d\times d$-matrix. Given a matrix $A$, denote by $A^*$ the transpose matrix of $A$.
Recall that for any $y\in \mathbb{R}^d$, $DF(y)$ denotes the Jacobian matrix of $F$ at $y$.

A sufficient condition that allows to uniformly push back to the origin the dynamics of \eqref{SEDO} is the following
{\em strong coercivity condition}:
there exists $\delta >0$ such that
\begin{equation}
\tag{H}
\label{C2}
\<x,DF(y)x\> \geq \delta \|x\|^2\quad \textrm{ for any } x,y \in \mathbb{R}^d.
\end{equation}
At the beginning of Section \ref{mulana} we will see that \eqref{C2} implies \eqref{C1}.
To control the growth  of the vector field $F$ around infinity, we 
assume the following {\em growth condition}:
there exist positive constants $c_0$ and $c_1$ such that 
\begin{equation}
\tag{G}
\label{C3}
\|F(x)\| \leq c_0e^{c_1\|x\|^2}\quad \textrm{ for any } x \in \mathbb{R}^d.
\end{equation}
In the case of a stochastic perturbation of a dynamical system
satisfying the strongly coercivity condition \eqref{C2} and  the growth condition \eqref{C3}
we prove thermalisation.
\begin{theorem}\label{main}
\label{t1}
Assume that 
 \eqref{C2} and   \eqref{C3} hold.
Let $\{x^\epsilon(t,x_0): t \geq 0\}$ be the solution of \eqref{SEDO}
and denote by $\mu^\epsilon$ the unique invariant probability measure for the evolution given by \eqref{SEDO}.
Denote by 
\[
d^{\epsilon}(t)=\ud_{\mathrm{TV}}(x^{\epsilon}(t,x_0),\mu^\epsilon) \quad
\textrm{for any } t\geq 0
\] 
the total variation distance between the law of the random variable $x^{\epsilon}(t,x_0)$ and its invariant probability $\mu^\epsilon$.
Consider the cut-off time $t_{\mathrm{mix}}^\epsilon$  given by \eqref{mix} and the time window given by
\eqref{win}. Let $x_0\not=0$. Then for
any $c\in \mathbb{R}$ we have
\[
\lim\limits_{\epsilon\rightarrow 0}{\Big|d^{\epsilon}(t_{\mathrm{mix}}^\epsilon+cw^\epsilon)-D^{\epsilon}(t_{\mathrm{mix}}^\epsilon+cw^\epsilon)\Big|}=0,
\]
where
\begin{equation}\label{gap}
D^{\epsilon}(t)=\ud_{\mathrm{TV}}\left(\mathcal{G}\left(\frac{(t-\tau)^{\ell-1}}{e^{\lambda (t-\tau)}\sqrt{\epsilon}}\Sigma^{-\nicefrac{1}{2}}\sum\limits_{k=1}^{m}e^{i\theta_k (t-\tau)}v_k,\mathrm{I}_d \right),\mathcal{G}(0,\mathrm{I}_d) \right)
\end{equation}
for any $t\geq \tau$
with $m$, $\lambda$, $\ell$, $\tau$, $\theta_1,\ldots, \theta_m $, $v_1,\ldots,v_m$ are the constants and vectors  associated to $x_0$
in {\em{Lemma \ref{asymp}}},
and the matrix $\Sigma$ is the unique solution of the matrix Lyapunov equation:
\begin{equation}
\label{lyapv}
DF(0)X+X (DF(0))^*=\mathrm{I}_d.
\end{equation}
\end{theorem}
\begin{remark}
The last theorem tells us that the total variation distance between the law of $x^{\epsilon}(t)$  and its equilibrium $\mu^{\epsilon}$ can be well approximated in a time window \eqref{win} around the cut-off time \eqref{mix} by the total variation distance between two Gaussian distributions \eqref{gap}.
\end{remark}
\begin{remark}\label{neg}
From Lemma \ref{a2} we deduce 
an ``explicit" formula for the distance \eqref{gap}, i.e.,
\[
D^{\epsilon}(t)=
\sqrt{\frac{2}{{\pi}}}\int\limits_{0}^{\nicefrac{\|m^\epsilon(t)\|}{2}}e^{-\frac{x^2}{2}}\ud x,
\]
where $m^{\epsilon}(t)=\frac{(t-\tau)^{\ell-1}}{e^{\lambda (t-\tau)}\sqrt{\epsilon}}\Sigma^{-\nicefrac{1}{2}}\sum\limits_{k=1}^{m}e^{i\theta_k (t-\tau)}v_k$ for any $t\geq \tau$.
\end{remark}

\begin{remark}\label{refad}
Since the linear differential equation 
\[\ud x(t)=-DF(0)x(t)\ud t \quad \textrm{ for any } t\geq 0 \] is asymptotically stable, then the matrix Lyapunov equation \eqref{lyapv} has a unique solution $\Sigma$ which is symmetric and positive definite and it is given by the formula:
\[
\Sigma=\int\limits_{0}^{\infty}e^{-DF(0)s}
e^{-(DF(0))^*s}
\ud s.
\]
For more details, see Theorem $1$, page $443$ of \cite{LATI}.
If in addition, $DF(0)$ is symmetric then $\Sigma$ is easily computable and it is given by $\Sigma=\frac{1}{2}(DF(0))^{-1}$.
\end{remark}

From Theorem \ref{main} we have the following consequences that we write as  corollaries.
To made the presentation more fluent,
in all the corollaries below, we will assume the same hypothesis of Theorem \ref{main} and keep the same notation. 
For any $\epsilon\in (0,1]$ and $x_0\in \mathbb{R}^d$, denote by $x^{\epsilon,x_0}$  the Markov process $\{x^{\epsilon}(t,x_0):t\geq 0\}$. 

\begin{corollary}\label{nuevo}
Suppose that $x_0\not=0$.
Window thermalisation for the distance $D^{\epsilon}$ at cut-off time  $t^\epsilon_{\mathrm{mix}}$ and time window $w^{\epsilon}$ is equivalent to window thermalisation for the distance $d^{\epsilon}$
at cut-off time $t^\epsilon_{\mathrm{mix}}$ and time window $w^{\epsilon}$. The same holds true for profile thermalisation.
\end{corollary}
\begin{proof}
It follows easily from Theorem \ref{main} and the following inequalities
\begin{equation*}
D^{\epsilon}(t^\epsilon_{\mathrm{mix}}+cw^{\epsilon})\leq |D^{\epsilon}(t^\epsilon_{\mathrm{mix}}+cw^{\epsilon})-d^{\epsilon}(t^\epsilon_{\mathrm{mix}}+cw^{\epsilon})|+d^{\epsilon}(t^\epsilon_{\mathrm{mix}}+cw^{\epsilon})
\end{equation*}
and
\begin{equation*}
d^{\epsilon}(t^\epsilon_{\mathrm{mix}}+cw^{\epsilon})\leq |D^{\epsilon}(t^\epsilon_{\mathrm{mix}}+cw^{\epsilon})-d^{\epsilon}(t^\epsilon_{\mathrm{mix}}+cw^{\epsilon})|+D^{\epsilon}(t^\epsilon_{\mathrm{mix}}+cw^{\epsilon}).
\end{equation*}
\end{proof}

\begin{corollary}[Thermalisation]
Suppose that $x_0\not=0$.
Theorem \ref{main} implies thermalisation  
for the 
 family of stochastic Markov processes 
 $\{x^{\epsilon,x_0}\}_{\epsilon\in (0,1]}$.
\end{corollary}

\begin{proof}
From Corollary \ref{nuevo}, we only need to analyse  the distance $D^{\epsilon}$.
Notice
\[0<L:=\liminf\limits_{t\rightarrow +\infty}\left\|\sum\limits_{k=1}^{m}e^{i\theta_k (t-\tau)}v_k\right\|\leq \limsup\limits_{t\rightarrow +\infty}\left\|\sum\limits_{k=1}^{m}e^{i\theta_k (t-\tau)}v_k\right\|\leq \sum\limits_{k=1}^{m}\|v_k\|=:U,\]
where first inequality follows from the Cantor diagonal argument and the fact that $v_1,\ldots,v_m$ are linearly independent. 
From Remark \ref{neg} we have
\[
D^{\epsilon}(t)=
\sqrt{\frac{2}{{\pi}}}\int\limits_{0}^{\nicefrac{\|m^{\epsilon}(t)\|}{2}}e^{-\frac{x^2}{2}}\ud x,
\]
where $m^{\epsilon}(t)=\frac{(t-\tau)^{\ell-1}}{e^{\lambda (t-\tau)}\sqrt{\epsilon}}\Sigma^{-\nicefrac{1}{2}}\sum\limits_{k=1}^{m}e^{i\theta_k (t-\tau)}v_k$ for any $t\geq \tau$.
Notice that
\[
Le^{-c}\leq \liminf\limits_{\epsilon \to 0}
\|m^{\epsilon}(t_{\mathrm{mix}}^\epsilon+cw^\epsilon)\|\leq 
\limsup\limits_{\epsilon \to 0}
\|m^{\epsilon}(t_{\mathrm{mix}}^\epsilon+cw^\epsilon)\|\leq Ue^{-c}
\]
for any $c\in \mathbb{R}$.
From Lemma \ref{lusc} and Lemma \ref{a2} we deduce
\[
\begin{split}
\sqrt{\frac{2}{{\pi}}}\int\limits_{0}^{\nicefrac{Le^{-c}}{2}}e^{-\frac{x^2}{2}}\ud x\leq & \liminf\limits_{\epsilon \to 0}
D^{\epsilon}(t_{\mathrm{mix}}^\epsilon+cw^\epsilon)\\
\leq &
\limsup\limits_{\epsilon \to 0}
D^{\epsilon}(t_{\mathrm{mix}}^\epsilon+cw^\epsilon)\leq \sqrt{\frac{2}{{\pi}}}\int\limits_{0}^{\nicefrac{Ue^{-c}}{2}}e^{-\frac{x^2}{2}}\ud x
\end{split}
\]
for any $c\in \mathbb{R}$.
Therefore
\[
\lim_{c\to -\infty}\liminf\limits_{\epsilon \to 0}
D^{\epsilon}(t_{\mathrm{mix}}^\epsilon+cw^\epsilon)=1 
\]
and
\[
\lim_{c\to +\infty}\limsup\limits_{\epsilon \to 0}
D^{\epsilon}(t_{\mathrm{mix}}^\epsilon+cw^\epsilon)=0.
\]
\end{proof}

\begin{remark}\label{pos}
Recall that $\{v_1,\ldots,v_m\}$ are linearly independent in $\mathbb{C}$. 
If in addition
\[\lim\limits_{t\rightarrow +\infty}\left\|\Sigma^{-\nicefrac{1}{2}}\sum_{k=1}^{m}e^{i\theta_k t}v_k\right\|\quad \textrm{ is well defined, }\]
then  
\[\lim\limits_{t\rightarrow +\infty}\left\|\Sigma^{-\nicefrac{1}{2}}\sum\limits_{k=1}^{m}e^{i\theta_k t}v_k\right\|=
\left\|\Sigma^{-\nicefrac{1}{2}}\sum\limits_{k=1}^{m}v_k\right\|>0.
\]
In this case, we define 
\[
r(x_0):=\left\|\Sigma^{-\nicefrac{1}{2}}\sum\limits_{k=1}^{m}v_k\right\|>0.
\]
\end{remark}

\begin{corollary}[Profile thermalisation]\label{pt}
Suppose that $x_0\not=0$.
There is profile thermalisation  for the 
 family of stochastic Marvok processes 
 $\{x^{\epsilon,x_0}\}_{\epsilon\in (0,1]}$ if and only if 
\[\lim\limits_{t\rightarrow +\infty}\left\|\Sigma^{-\nicefrac{1}{2}}\sum\limits_{k=1}^{m}e^{i\theta_k t}v_k\right\|\quad \textrm{ is well defined}.\]
\end{corollary}

\begin{proof}
Suppose that there is profile thermalisation for
 $\{x^{\epsilon,x_0}\}_{\epsilon\in (0,1]}$.
 Then $\lim\limits_{\epsilon \to 0}d^{\epsilon}(t_{\mathrm{mix}}^\epsilon+cw^\epsilon)$  exists for any $c\in \mathbb{R}$. From Corollary \ref{nuevo} we have that
 $\lim\limits_{\epsilon \to 0}D^{\epsilon}(t_{\mathrm{mix}}^\epsilon+cw^\epsilon)$  also exists for any $c\in \mathbb{R}$.
From Remark \ref{neg} we deduce 
\[\lim\limits_{t\rightarrow +\infty}\left\|\Sigma^{-\nicefrac{1}{2}}\sum\limits_{k=1}^{m}e^{i\theta_k t}v_k\right\|\quad \textrm{ is well defined}.\]
On the other hand, if 
$\lim\limits_{t\rightarrow +\infty}\left\|\Sigma^{-\nicefrac{1}{2}}\sum\limits_{k=1}^{m}e^{i\theta_k t}v_k\right\|=r(x_0)$, from Remark \ref{pos} we have that $r(x_0)>0$ and from Remark \ref{neg} we get 
\begin{equation*}\label{lir}
\lim\limits_{\epsilon \to 0}D^{\epsilon}(t_{\mathrm{mix}}^\epsilon+cw^\epsilon)=
\sqrt{\frac{2}{\pi}}\int\limits_{0}^{\frac{e^{-c}r(x_0)}{2}}e^{-\frac{x^2}{2}}\ud x \quad \textrm{ for any } c\in \mathbb{R}.
\end{equation*}
The latter together with Corollary \ref{nuevo} imply profile thermalisation for $\{x^{\epsilon,x_0}\}_{\epsilon\in (0,1]}$.
\end{proof}

The following corollary includes the case when the dynamics is reversible, {\em{i.e.}}, $F=\nabla V$ for some scalar function $V:\mathbb{R}^d\rightarrow \mathbb{R}$. 
\begin{corollary}\label{reals}
Suppose that $x_0\not=0$. If all the eigenvalues of $DF(0)$ are real then the 
 family of stochastic processes 
 $\{x^{\epsilon,x_0}\}_{\epsilon\in (0,1]}$ has profile thermalisation.
\end{corollary}
\begin{proof}
The proof follows from Corollary \ref{pt} observing that $\theta_j=0$ for any $j=1,\ldots,m$ and the fact that 
 $\{v_1,\ldots,v_m\}$ are linearly independent in $\mathbb{C}$.
\end{proof}

Moreover, in \cite{BJ}, we study the case when $d=1$
which follows immediately from  Corollary \ref{reals}.

We also have a {\bf{dynamical characterisation}} of profile thermalisation.
Define ``a normalised" $\omega$-limit set of $x_0$ as follows:
\begin{align*}
\omega(x_0):=& \left\{y\in \mathbb{R}^d: \textrm{ there exists a sequence of positive numbers } \{t_n: n\in \mathbb{N}\}\right.\\
& \left. \textrm{ such that }  \lim\limits_{n\rightarrow +\infty}t_n=+\infty 
 \textrm{ and }
 \lim\limits_{n\rightarrow +\infty} \frac{e^{\lambda t_n}}{t_n^{\ell-1}}\Sigma^{-\nicefrac{1}{2}}e^{-DF(0)t_n}x_0=y   \right\}.
\end{align*}
From Lemma \ref{asy}, it is not hard to see that
$\Sigma^{-\nicefrac{1}{2}}\sum_{k=1}^{m}v_k\in \omega(x_0)$. 
When all the eigenvalues of $DF(0)$ are real, then again by  Lemma \ref{asy}, we get that $\omega(x_0)$ consists of a non-zero element which is given by $\Sigma^{-\nicefrac{1}{2}}\sum_{k=1}^{m}v_k$.
\begin{corollary}\label{dpt}
Suppose that $x_0\not=0$. The 
 family of stochastic Markov processes 
 $\{x^{\epsilon,x_0}\}_{\epsilon\in (0,1]}$ has profile thermalisation if and only if
$\omega(x_0)$ is contained in a $d$-sphere with radius
$r(x_0):=\left\|\Sigma^{-\nicefrac{1}{2}}\sum\limits_{k=1}^{m}v_k\right\|$, i.e.,
$\omega(x_0)\subset \mathbb{S}^{d-1}{(r(x_0))}$,
where $\mathbb{S}^{d-1}{(r(x_0))}:=\left\{x\in \mathbb{R}^d: \|x\|=r(x_0)\right\}$.
\end{corollary}
\begin{proof}
Suppose that $\{x^{\epsilon,x_0}\}_{\epsilon\in(0,1]}$ has profile thermalisation. By Corollary
\ref{pt} we have  
\[
\lim_{t \to +\infty} \Big\|\Sigma^{-\nicefrac{1}{2}} \sum_{k=1}^m e^{i\theta_k t} v_k \Big\|\quad  \textrm{ is well defined}.
\]
From Remark \ref{pos} we know
\[
\lim_{t \to +\infty} \Big\|\Sigma^{-\nicefrac{1}{2}} \sum_{k=1}^m e^{i\theta_k t} v_k \Big\|
=
r(x_0)>0.
\]
The latter together with Lemma \ref{asy} allows to deduce that
\[
\lim_{t \to +\infty} \Big\| \frac{e^{\lambda t}}{t^{\ell-1}}\Sigma^{-\nicefrac{1}{2}}e^{-DF(0)t}x_0 \Big\|=
r(x_0).
\]
Consequently, $\omega(x_0)\subset \mathbb{S}^{d-1}\left(r(x_0)\right)$.

On the other hand, suppose that $\omega(x_0)\subset \mathbb{S}^{d-1}(r(x_0))$. 
Then 
\[
\lim_{t \to +\infty} 
\Big\| \frac{e^{\lambda t}}{t^{\ell-1}} e^{-DF(0)t}x_0 \Big\|=r(x_0).
\]
From Lemma \ref{asy} we get
\[
\lim_{t \to +\infty} \Big\|\Sigma^{-\nicefrac{1}{2}} \sum_{k=1}^m e^{i\theta_k t} v_k \Big\|=r(x_0).
\]
The latter together with Corollary \ref{pt} allow us to deduce the statement.
\end{proof}

In dimension $2$ and $3$, we can state a {\bf{spectral characterisation}} of profile thermalisation. Remind that if all the eigenvalues of $DF(0)$ are real, we have profile thermalisation as Corollary \ref{reals} stated, so we avoid that case.
\begin{corollary}
\label{dim2}
Suppose that $x_0\not=0$ and $d=2$.
Let $\gamma$ be a complex eigenvalue of $DF(0)$ with non-zero imaginary part and let $u_1+iu_2$ be its eigenvector, where $u_1,u_2\in \mathbb{R}^2$.
Then the 
 family of Markov stochastic processes 
 $\{x^{\epsilon,x_0}\}_{\epsilon\in (0,1]}$ has profile thermalisation if and only if
$\<u_1,\Sigma^{-1}u_1\>=\<u_2,\Sigma^{-1}u_2\>$ and $\<u_1,\Sigma^{-1}u_2\>=0$.
\end{corollary}

\begin{proof}
Write $\gamma=\lambda+i\theta$, where $\lambda>0$ with 
$\theta \not=0$. To the eigenvalue $\gamma$ we associated an eigenvector 
$u_1+iu_2$, where  $u_1,u_2\in \mathbb{R}^2$.
An straightforward computation shows 
\[
e^{\lambda t}e^{-DF(0)t}x_0=
(c_1\cos(\theta t)-c_2\sin(\theta t))u_1+(c_1\sin(\theta t)+c_2\cos(\theta t))u_2 
\]
for any $t\geq 0$, where 
$c_1:=c_1(x_0)$ and $c_2:=c_2(x_0)$ are not both zero.
Notice that $c:=\sqrt{c^2_1+c^2_2}>0$ and let $\cos(\alpha)=\nicefrac{c_1}{c}$ and
$\sin(\alpha)=\nicefrac{c_2}{c}$. Then
\[
e^{\lambda t}e^{-DF(0)t}x_0=
c\cos(\theta t+\alpha)u_1+c\sin(\theta t+\alpha)u_2 \quad \textrm{ for any } t\geq 0.
\]
Therefore, 
\begin{equation}\label{lop}
\begin{split}
\|\Sigma^{-\nicefrac{1}{2}}e^{\lambda t}e^{-DF(0)t}x_0\|^2 & =
c^2\cos^2(\theta t+\alpha)\<u_1,\Sigma^{-1}u_1\>\\
&\hspace{-3cm}+
c^2\sin^2(\theta t+\alpha)\<u_2,\Sigma^{-1}u_2\>+
2c^2\cos(\theta t+\alpha)\sin(\theta t+\alpha)\<u_1,\Sigma^{-1}u_2\>
\end{split}
\end{equation}
for any $t\geq 0$.
If $\<u_1,\Sigma^{-1}u_1\>=\<u_2,\Sigma^{-1}u_2\>$ and 
$\<u_1,\Sigma^{-1}u_2\>=0$ then 
\[
\|\Sigma^{-\nicefrac{1}{2}}e^{\lambda t}e^{-DF(0)t}x_0\|^2  =
c^2\<u_1,\Sigma^{-1}u_1\> \quad \textrm{ for any } t\geq 0.
\]
Notice that $\<u_1,\Sigma^{-1}u_1\>\not=0$ since $u_1\not=0$ and $\Sigma^{-1}$ is a positive definite symmetric matrix.
The conclusion follows easily from Lemma \ref{asy} and Corollary \ref{pt}. 

On the other hand,
if  $\{x^{\epsilon,x_0}\}_{\epsilon\in (0,1]}$ has profile thermalisation then Lemma \ref{asy} and Corollary \ref{pt} imply 
\[
\lim\limits_{t\to +\infty}\|\Sigma^{-\nicefrac{1}{2}}e^{\lambda t}e^{-DF(0)t}x_0\|^2  \quad \textrm{ is well defined.}
\]
Now, using \eqref{lop} and taking different subsequences we deduce 
$\<u_1,\Sigma^{-1}u_1\>=\<u_2,\Sigma^{-1}u_2\>$ and
$\<u_1,\Sigma^{-1}u_2\>=0$.
\end{proof}

Notice that in dimension $3$, at least one eigenvalue of $DF(0)$ is real. 
Therefore the interesting case is when the others eigenvalues are complex numbers with non-zero imaginary part.
Let $\gamma_1$ be a real eigenvalue of $DF(0)$ with eigenvector $v\in \mathbb{R}^3$.
Let $\gamma$ be a complex eigenvalue of $DF(0)$ with non-zero imaginary part and let $u_1+iu_2$ be its eigenvector, where $u_1,u_2\in \mathbb{R}^3$. In this case,
\[
\begin{split}
e^{-DF(0)t}x_0  = & c(x_0)e^{-\gamma_1 t}v\\
&+e^{-\lambda t}(c_1\cos(\theta t)-c_2\sin(\theta t))u_1+(c_1\sin(\theta t)+c_2\cos(\theta t))u_2) 
\end{split}
\]
for any $t\geq 0$, where 
$c_0:=c_0(x_0)$, $c_1:=c_1(x_0)$ and $c_2:=c_2(x_0)$ are not all zero. Notice that $c:=\sqrt{c^2_1+c^2_2}>0$ and let $\cos(\alpha)=\nicefrac{c_1}{c}$ and
$\sin(\alpha)=\nicefrac{c_2}{c}$. Then
\begin{equation}\label{bnm}
\begin{split}
e^{-DF(0)t}x_0  = & c_0e^{-\gamma_1 t}v
+e^{-\lambda t}(c\cos(\theta t+\alpha)u_1+c\sin(\theta t+\alpha)u_2) 
\end{split}
\end{equation}
for any $t\geq 0$.

\begin{corollary}
\label{dim3}
Suppose that $x_0\not=0$ and $d=3$.
Let $\gamma_1$ be a real eigenvalue of $DF(0)$ with eigenvector $v\in \mathbb{R}^3$.
Let $\gamma$ be a complex eigenvalue of $DF(0)$ with non-zero imaginary part and let $u_1+iu_2$ be its eigenvector, where $u_1,u_2\in \mathbb{R}^3$. Let $c_0$, $c$ and $\alpha$ the constants that appears in \eqref{bnm}. 
\begin{itemize}
\item[i)]  Assume $c_0=0$. 
 $\{x^{\epsilon,x_0}\}_{\epsilon\in (0,1]}$ has profile thermalisation if and only if
$\<u_1,\Sigma^{-1}u_1\>=\<u_2,\Sigma^{-1}u_2\>$ and $\<u_1,\Sigma^{-1}u_2\>=0$.
\item[ii)] Assume $c_0\not=0$ and $\gamma_1< \lambda$.
Then the family $\{x^{\epsilon,x_0}\}_{\epsilon\in (0,1]}$ has profile thermalisation.
\item[iii)] Assume $c_0\not=0$, $\gamma_1=\lambda$. 
The family $\{x^{\epsilon,x_0}\}_{\epsilon\in (0,1]}$ has profile thermalisation if and only if
$\<u_1,\Sigma^{-1}u_1\>=\<u_2,\Sigma^{-1}u_2\>$ and $\<u_1,\Sigma^{-1}u_2\>=\<v,\Sigma^{-1}u_1\>=\<v,\Sigma^{-1}u_2\>=0$.
\item[iv)] Assume $c_0\not=0$, $\gamma_1>\lambda$. 
The family $\{x^{\epsilon,x_0}\}_{\epsilon\in (0,1]}$ has profile thermalisation if and only if
$\<u_1,\Sigma^{-1}u_1\>=\<u_2,\Sigma^{-1}u_2\>$ and $\<u_1,\Sigma^{-1}u_2\>=0$.
\end{itemize} 
\end{corollary}

\begin{proof}~
\begin{itemize}
\item[i)] This case can be deduced using the same arguments as Corollary \ref{dim2}.
\item[ii)] Using relation \eqref{bnm} we obtain
\[
\lim\limits_{t\rightarrow +\infty} \Sigma^{-\nicefrac{1}{2}}e^{\gamma_1 t}e^{-DF(0)t}x_0=c_0\Sigma^{-\nicefrac{1}{2}}v\not=0.
\]
The latter together with Corollary \ref{pt} allows us to deduce profile thermalisation.
\item[iii)] Using relation \eqref{bnm} for any $t\geq 0$ we get
\begin{equation}\label{xcv}
\begin{split}
\|\Sigma^{-\nicefrac{1}{2}}e^{\gamma_1 t}e^{-DF(0)t}x_0\|^2=&c^2_0\<v,\Sigma^{-\nicefrac{1}{2}}v\>+c^2\cos^2(\theta t+\alpha)\<u_1,\Sigma^{-\nicefrac{1}{2}}u_1\>\\
&\hspace{-4cm}+c^2\sin^2(\theta t+\alpha)\<u_2,\Sigma^{-\nicefrac{1}{2}}u_2\>
+2c_0c\cos(\theta t+\alpha)\<v,\Sigma^{-\nicefrac{1}{2}}u_1\>\\
&\hspace{-4cm}+2c_0c\sin(\theta t+\alpha)\<v,\Sigma^{-\nicefrac{1}{2}}u_2\>
+2c^2\cos(\theta t+\alpha)\sin(\theta t+\alpha)\<u_1,\Sigma^{-\nicefrac{1}{2}}u_2\>.
\end{split}
\end{equation}
If
$\<u_1,\Sigma^{-1}u_1\>=\<u_2,\Sigma^{-1}u_2\>$ and 
\[\<u_1,\Sigma^{-1}u_2\>=\<v,\Sigma^{-1}u_1\>=\<v,\Sigma^{-1}u_2\>=0,\] then from \eqref{xcv} we deduce 
\[
\|\Sigma^{-\nicefrac{1}{2}}e^{\gamma_1 t}e^{-DF(0)t}x_0\|=
c^2_0\<v,\Sigma^{-\nicefrac{1}{2}}v\>+c^2\<u_1,\Sigma^{-1}u_1\>>0.
\]
The latter together with Lemma \ref{asy} and Corollary \ref{pt} imply profile thermalisation.

On the other hand,
if  $\{x^{\epsilon,x_0}\}_{\epsilon\in (0,1]}$ has profile thermalisation then Lemma \ref{asy} and Corollary \ref{pt} imply 
\[
\lim\limits_{t\to +\infty}\|\Sigma^{-\nicefrac{1}{2}}e^{\lambda t}e^{-DF(0)t}x_0\|^2  \quad \textrm{ is well defined.}
\]
Now, using  \eqref{xcv} and taking different subsequences we deduce
\[\<u_1,\Sigma^{-1}u_1\>=\<u_2,\Sigma^{-1}u_2\>\]
 and 
\[\<u_1,\Sigma^{-1}u_2\>=\<v,\Sigma^{-1}u_1\>=\<v,\Sigma^{-1}u_2\>=0.\]
\item[iv)] 
Using relation \eqref{bnm} for any $t\geq 0$ we have
\begin{equation}\label{xcv1}
\begin{split}
\|\Sigma^{-\nicefrac{1}{2}}e^{\lambda t}e^{-DF(0)t}x_0\|^2=&c^2_0e^{-2(\gamma_1-\lambda)t}\<v,\Sigma^{-\nicefrac{1}{2}}v\>+c^2\cos^2(\theta t+\alpha)\<u_1,\Sigma^{-\nicefrac{1}{2}}u_1\>\\
&\hspace{-3cm}+c^2\sin^2(\theta t+\alpha)\<u_2,\Sigma^{-\nicefrac{1}{2}}u_2\>
+2c_0ce^{-(\gamma_1-\lambda)t}\cos(\theta t+\alpha)\<v,\Sigma^{-\nicefrac{1}{2}}u_1\>\\
&\hspace{-3cm}+2c_0ce^{-(\gamma_1-\lambda)t}\sin(\theta t+\alpha)\<v,\Sigma^{-\nicefrac{1}{2}}u_2\>\\
&\hspace{-3cm}
+2c^2\cos(\theta t+\alpha)\sin(\theta t+\alpha)\<u_1,\Sigma^{-\nicefrac{1}{2}}u_2\>.
\end{split}
\end{equation}
If
$\<u_1,\Sigma^{-1}u_1\>=\<u_2,\Sigma^{-1}u_2\>$ and $\<u_1,\Sigma^{-1}u_2\>=0$  from \eqref{xcv1} we obtain
\[
\lim\limits_{t\to +\infty}
\|\Sigma^{-\nicefrac{1}{2}}e^{\lambda t}e^{-DF(0)t}x_0\|=
c^2\<u_1,\Sigma^{-1}u_1\>\not=0
\]
which together with Corollary \ref{pt} imply profile thermalisation.

On the other hand,
if  $\{x^{\epsilon,x_0}\}_{\epsilon\in (0,1]}$ has profile thermalisation then Lemma \ref{asy} and Corollary \ref{pt} imply 
\[
\lim\limits_{t\to +\infty}\|\Sigma^{-\nicefrac{1}{2}}e^{\lambda t}e^{-DF(0)t}x_0\|^2  \quad \textrm{ is well defined.}
\]
Now, using  \eqref{xcv1} and taking different subsequences one can deduce 
$\<u_1,\Sigma^{-1}u_1\>=\<u_2,\Sigma^{-1}u_2\>$ and 
$\<u_1,\Sigma^{-1}u_2\>=0.$
\end{itemize}
\end{proof}

When $\Sigma$ is the identity matrix, roughly speaking Corollary \ref{dim2} and Corollary \ref{dim3} state that profile thermalisation is equivalent to ``norm" preserving and orthogonality of the real and imaginary parts of the eigenvectors of $DF(0)$. When $\Sigma$ is not the identity, the latter still true under a change of basis. 

\section{The multiscale analysis}\label{mulana}
In this section, we prove that the process $\{x^\epsilon(t): t \geq 0\}$ can be well approximated by the solution of a 
linear non-homogeneous process in a time window that will include the time scale on which we are interested. 
It is not hard to see that
 \eqref{C2} basically says that  \eqref{C1} is satisfied around any point $y$. In fact, writing
\[
F(y)-F(x) = \int_0^1 \tfrac{\ud}{\ud t} F(x+t(y-x)) \ud t = \int_0^1 DF(x+t(y-x))(y-x)\ud t
\]
we obtain the seemingly stronger condition
\begin{equation*}
\label{C4}
\<y-x,F(y)-F(x)\> \geq \delta \|y-x\|^2\quad \text{ for any } x,y \in \bb R^d.
\end{equation*}
The latter is basically saying that \eqref{C1} is satisfied around any point $y \in \bb R^d$.
A good example of a vector field $F$ satisfying \eqref{C2} and \eqref{C3} is $F(x) = Ax + H(x)$, $x\in \mathbb{R}^d$,
where $A$ is a matrix,
$H$ is a vector valued function such that $F$ satisfies \eqref{C2} and
it satisfies $H(0)=0$, $DH(0)=0$, $\|DH\|_\infty <+\infty$ and $\|D^2H\|_\infty <+\infty$.
In dimension one, a good example to keep
 in mind is 
\begin{equation}\label{ejemplo}
F(x)=\sum\limits_{j=1}^{n}a_jx^{2j-1}\quad \textrm{ for any } x\in \mathbb{R},
\end{equation}
where $n\in \mathbb{N}$, $a_1>0$ and $a_j\geq 0$ for any $j\in \{2,\ldots,n\}$.
It is fairly easy to see that \eqref{ejemplo} satisfies \eqref{C2} and \eqref{C3}.
If $a_j>0$ for some $j\in \{2,\ldots,n\}$ then \eqref{ejemplo} is not globally Lipschitz continuous.

Recall that $F\in \mathcal{C}^2(\mathbb{R}^d,\mathbb{R}^d)$. Notice that
for any $x,y \in \mathbb{R}^d$ we have
\[
F(x)-F(y)=\int_0^1 DF(x+t(y-x))(y-x)\ud t.
\]
Therefore, for any $x,y \in \mathbb{R}^d$ we get
\[
F(x)-F(y)-DF(y)(x-y)=\int_0^1 (DF(y+t(x-y))-DF(y))(x-y)\ud t. 
\]
Note
\[
DF(y+t(x-y))-DF(y)=\int_0^1 D^2F(y+st(x-y))t(x-y)\ud s 
\]
for any 
$x,y \in \mathbb{R}^d$.
For any $r_0>0$ and $r_1>0$, define $C:=\sup\limits_{|z|\leq 2r_1+r_0}\|D^2F(z)\|$. Then
\begin{equation}\label{cota2}
\|F(x)-F(y)-DF(y)(x-y)\|\leq C\|x-y\|^2
\end{equation}
for any $\|x\|\leq r_0$ and $\|y\|\leq r_1$. Inequality \eqref{cota2} will allow us to control 
the random dynamics $\{x^{\epsilon}(t):t\geq 0\}$ on compacts sets and it will be very useful in our {\it{apriori}} estimates.

\subsection{Zeroth-order approximations}
It is fairly easy to see that for any $t\geq 0$, as $\epsilon \to 0$, $x^\epsilon(t)$ converges to $\varphi(t)$. The convergence can be proved to be almost surely uniform in compacts. But for our purposes, we need a {\em quantitative estimate} 
on the distance between $x^\epsilon(t)$ and 
$\varphi(t)$. The idea is fairly simple: \eqref{C2} says that the dynamical system \eqref{EDO} is uniformly contracting. Therefore, it is reasonable that fluctuations are pushed back to the solution of \eqref{EDO} and therefore the difference between $x^\epsilon(t)$ and $\varphi(t)$ has a short-time dependence on the noise $\{B(s):0\leq s\leq t\}$. This heuristics can be made precise computing the It\^o derivative of
$\|x^\epsilon(t) -\varphi(t)\|^2$ as follows:

\begin{align}
\ud \|x^\epsilon(t) -\varphi(t)\|^2
		&= -2 \<x^\epsilon(t) -\varphi(t), F(x^\epsilon(t)) -F(\varphi(t))\>\ud t
		\nonumber\\
 &\quad+2\sqrt \epsilon \<x^\epsilon(t) -\varphi(t), \ud B(t)\> +  d \epsilon \ud t \nonumber\\
		&\leq -2\delta \|x^\epsilon(t) -\varphi(t)\|^2\ud t +2\sqrt \epsilon \<x^\epsilon(t) -\varphi(t), \ud B(t)\>+ d \epsilon \ud t, \label{pin} 
\end{align}
where the last inequality follows from \eqref{C2}.
After a localisation argument we get
\begin{align*}
\frac{\ud}{\ud t} \mathbb{E}\left[\|x^\epsilon(t) -\varphi(t)\|^2\right]
\leq -2\delta \mathbb{E}\left[\|x^\epsilon(t) -\varphi(t)\|^2\right]+\epsilon d
\quad \textrm{ for any } t\geq 0.
\end{align*}
From Lemma \ref{gin}, we obtain the following uniform bound
\begin{equation}
\label{zeroth}
\bb E \left[\|x^\epsilon(t) -\varphi(t)\|^2 \right] \leq \frac{d \epsilon}{2\delta}(1-e^{-2\delta t})\leq \frac{d \epsilon}{2\delta} \quad\textrm{ for any } t\geq 0.
\end{equation}
We call this bound the {\em zeroth order} approximation of $x^\epsilon(t)$.
We have just proved that the distance between $x^\epsilon(t)$ and $\varphi(t)$ is of order $\mathcal{O}(\sqrt \epsilon)$, uniformly in $t\geq 0$. However, this estimate is meaningful only while $\|\varphi(t)\| \gg \sqrt \epsilon$. By Lemma \ref{asymp}, $\|\varphi(t)\|$ is of order $\mathcal{O}(t^{\ell-1} e^{-\lambda t})$, which means that \eqref{zeroth} is meaningful for times $t$ of order $o(t^\epsilon_{\mathrm{mix}})$, which fall just short of what we need. This is very natural, because at times of order $t_{\mathrm{mix}}^\epsilon$ we expect that fluctuations play a predominant role.

\subsection{First-order approximations}
Notice  that \eqref{zeroth} can be seen as a law of large numbers for $x^\epsilon(t)$. In fact, $\bb E[ x^\epsilon(t)] = \varphi(t)$ for every $t\geq 0$ and for $t \ll t_{\mathrm{mix}}^\epsilon$, $\epsilon /\|\varphi(t)\|^2 \to 0$. By the second-moment method, $x^\epsilon(t)$ satisfies a law of large numbers when properly renormalised. Therefore, it is natural to look at the corresponding central limit theorem. Define $\{y^\epsilon(t): t \geq 0\}$ as
\[
y^\epsilon(t) = \frac{x^\epsilon(t) -\varphi(t)}{\sqrt \epsilon} \quad\text{ for any } t \geq 0.
\]
As above, it is not very difficult to prove that for every $T>0$, the process $\{y^\epsilon(t): t \in [0,T]\}$ converges in distribution to the solution $\{y(t): t \in [0,T]\}$ of the linear non-homogeneous stochastic differential equation (also known as non-homogeneous Ornstein-Uhlenbeck process):
\begin{equation}\label{dc}
\left\{
\begin{array}{r@{\;=\;}l}
\ud y(t) & -DF(\varphi(t)) y(t)\ud t + \ud B(t) \quad \textrm{ for }t\geq 0,\\
y(0) & 0.
\end{array}
\right.
\end{equation}
Notice that this equation is linear and in particular $y(t)$ has a Gaussian law for any $t> 0$.
As in the previous section, our aim is to obtain good quantitative bounds for the distance between $y^\epsilon(t)$ and $y(t)$. First, we notice that the estimate \eqref{zeroth} can be rewritten as
\begin{equation}
\label{zoa}
\bb E\left[\|y^\epsilon(t)\|^2\right] \leq \frac{d}{2\delta} \quad \text{ for any } t \geq 0.
\end{equation}
We will also need an upper bound for $\bb E\left[\|y^\epsilon(t)\|^4\right]$. From the It\^o formula and \eqref{C2} we have
\[
\begin{split}
\ud \|y^\epsilon(t)\|^4
		&= -4\|y^\epsilon(t)\|^2\<y^\epsilon(t), DF(\varphi(t)) y^\epsilon(t)\> \ud t
			+ 4\|y^\epsilon(t)\|^2 \<y^\epsilon(t), \ud B(t)\>\\
&\quad + (2d+4)\|y^\epsilon(t)\|^2 \ud t\\
		&\leq -4\delta \|y^\epsilon(t)\|^4\ud t + 4\|y^\epsilon(t)\|^2 \<y^\epsilon(t), \ud B(t)\> + (2d+4) \|y^\epsilon(t)\|^2 \ud t.
\end{split}
\]
After a localisation argument we obtain
\[
\tfrac{\ud}{\ud t} \bb E \left[\|y^\epsilon(t) \|^4\right] \leq -4\delta \bb E\left[\|y^\epsilon(t)\|^4\right] +(2d+4) \bb E\left[\|y^\epsilon(t)\|^2\right].
\]
From \eqref{zoa} and Lemma \ref{gin} we get the uniformly bound
\begin{equation}
\label{zoa1}
\bb E\left[\|y^\epsilon(t)\|^4\right] \leq \frac{d(d+2)}{4\delta^2}\big(1-e^{-4\delta t}\big) \leq \frac{d(d+2)}{4\delta^2}
\quad \text{ for any } t \geq 0.
\end{equation}
Notice that $x^{\epsilon}(t)=\varphi(t)+\sqrt{\epsilon}y^{\epsilon}(t)$ for any $t\geq 0$ and the difference $y^\epsilon(t)-y(t)$ has bounded variation. Then
\begin{align*}
\tfrac{\ud}{\ud t}(y^\epsilon(t) -y(t))
		&= -\frac{1}{\sqrt \epsilon}\big(F(x^\epsilon(t))-F(\varphi(t)) -\sqrt \epsilon DF(\varphi(t))y(t)\big) \\
		&= -\frac{1}{\sqrt \epsilon}\big( F(\varphi(t)+\sqrt{\epsilon}y^{\epsilon}(t))-F(\varphi(t)+\sqrt{\epsilon}y(t))\big)\\
        &\quad -\frac{1}{\sqrt \epsilon}\big( F(\varphi(t)+\sqrt{\epsilon}y(t)) - F(\varphi(t)) - \sqrt{\epsilon} DF(\varphi(t))y(t)\big).
\end{align*}
Define $h^\epsilon(t):=F(\varphi(t)+\sqrt{\epsilon}y(t)) - F(\varphi(t)) - \sqrt{\epsilon} DF(\varphi(t))y(t)$ for any $t\geq 0$.
Therefore, using the chain rule for $\|y^\epsilon(t) -y(t)\|^2$ we obtain 
the differential equation:
\begin{align}\label{tio}
&\tfrac{\ud}{\ud t} \| y^\epsilon(t) -y(t)\|^2 = 
2\<y^\epsilon(t) -y(t),\tfrac{\ud}{\ud t}(y^\epsilon(t) -y(t))\>= \nonumber\\
&-\frac{2}{\sqrt \epsilon}\<y^\epsilon(t) -y(t),F(\varphi(t)+\sqrt{\epsilon}y^{\epsilon}(t))-F(\varphi(t)+\sqrt{\epsilon}y(t))\>\nonumber\\
&-\frac{2}{\sqrt \epsilon}\<y^\epsilon(t) -y(t),h^\epsilon(t)\>\leq -2\delta \|y^\epsilon(t) -y(t)\|^2-\frac{2}{\sqrt \epsilon}\<y^\epsilon(t) -y(t),h^\epsilon(t)\>,
\end{align}
where the last inequality follows from \eqref{C2}.
From the Cauchy-Schwarz inequality we observe
\begin{align}\label{cz12}
|(\nicefrac{2}{\sqrt \epsilon})\<y^\epsilon(t) -y(t),h^\epsilon(t)\>|
\leq (\nicefrac{2}{\sqrt \epsilon})\|y^\epsilon(t) -y(t)\| \|h^{\epsilon}(t)\|.
\end{align}
Recall the well known Young type inequality $2|ab|\leq \varrho a^2+(\nicefrac{1}{\varrho})b^2$ for any $a,b\in \mathbb{R}$ and $\rho>0$. 
From inequality \eqref{cz12} we have
\begin{align}\label{cz1}
|(\nicefrac{2}{\sqrt \epsilon})\<y^\epsilon(t) -y(t),h^\epsilon(t)\>|
\leq  \delta\|y^\epsilon(t) -y(t)\|^2+ \frac{1}{\epsilon\delta}\|h^{\epsilon}(t)\|^2.
\end{align}
From inequality \eqref{tio} and inequality \eqref{cz1} we deduce
\begin{align*}
\tfrac{\ud}{\ud t} \| y^\epsilon(t) -y(t)\|^2 \leq 
-\delta \| y^\epsilon(t) -y(t)\|^2+\frac{1}{\epsilon\delta}\|h^{\epsilon}(t)\|^2.
\end{align*}
By taking expectation in both sides of the last inequality, we obtain
\begin{align}\label{cop1}
\tfrac{\ud}{\ud t} \bb E \left[\| y^\epsilon(t) -y(t)\|^2\right] \leq 
-\delta \bb E \left[ \| y^\epsilon(t) -y(t)\|^2\right]+\frac{1}{\epsilon\delta}\bb E  \left[\|h^{\epsilon}(t)\|^2\right].
\end{align}
Define $H^{\epsilon}(t):=\bb E \left[h^{\epsilon}(t)\|^2\right]$ for any $t\geq 0$.
From inequality \eqref{cop1} and Lemma \ref{gin} we deduce
\[
\bb E \left[\| y^\epsilon(t) -y(t)\|^2\right]\leq 
\frac{(1-e^{-\delta t})}{\epsilon \delta^2}\int_{0}^{t}H^{\epsilon}(s)\ud s
\quad \textrm{ for any } t\geq 0.
\]
Therefore, we need to get an upper bound for $\int_{0}^{t}H^{\epsilon}(s)\ud s
\textrm{ for any } t\geq 0$. From Lemma \ref{falta} we have
\[
\int\limits_{0}^{t}H^{\epsilon}(s)\ud s\leq C(\eta,\|x_0\|,d,\delta)\epsilon^2 t+\tilde{C}(\eta,\|x_0\|,d,\delta)\epsilon^{\nicefrac{3}{2}}t^{\nicefrac{7}{4}} 
\]
for any $\eta>0$ and $t\in \left[0,\frac{\eta^2}{2\epsilon d }\right)$,
where $C(\eta,\|x_0\|,d,\delta)$ and $\tilde{C}(\eta,\|x_0\|,d,\delta)$ are positive constants that only depend on $\eta$, $\|x_0\|$, $d$ and $\delta$.
The latter implies 
\begin{equation}\label{pario}
\begin{split}
\bb E \left[\| x^\epsilon(t) -(\varphi(t)+\sqrt{\epsilon}y(t))\|^2\right]\leq &\\
&\hspace{-2cm}\frac{1}{\delta^2}\left(C(\eta,\|x_0\|,d,\delta)\epsilon^2 t+\tilde{C}(\eta,\|x_0\|,d,\delta)\epsilon^{\nicefrac{3}{2}}t^{\nicefrac{7}{4}}\right)
\end{split}
\end{equation}
for any $t\in \left[0,\frac{\eta^2}{2\epsilon d }\right)$. We call this bound the {\em first-order} approximation of $x^\epsilon(t)$.
Roughly speaking, for $t=\mathcal{O}(\ln(\nicefrac{1}{\epsilon}))$ we have just proved that the distance between $x^\epsilon(t)$ and $\varphi(t)+\sqrt{\epsilon} y(t)$ is of order 
$\mathcal{O}\left(\epsilon^{\nicefrac{3}{4}-\wp}\right)$ for any $\wp\in (0,\nicefrac{3}{4})$ which will be enough for our purposes.
\begin{lemma}\label{falta}
Assume that 
 \eqref{C2} and  \eqref{C3} hold. Let $\epsilon\in \left(0,\frac{\delta}{32c_1}\right)$.
For any $\eta>0$ and $t\in \left[0,\frac{\eta^2}{2\epsilon d }\right)$ we have
\[
\int\limits_{0}^{t}H^{\epsilon}(s)\ud s\leq C(\eta,\|x_0\|,d,\delta)\epsilon^2 t+\tilde{C}(\eta,\|x_0\|,d,\delta)\epsilon^{\nicefrac{3}{2}}t^{\nicefrac{7}{4}}, 
\]
where $C(\eta,\|x_0\|,d,\delta)$ and $\tilde{C}(\eta,\|x_0\|,d,\delta)$ only depend on $\eta$, $\|x_0\|$, $d$ and $\delta$.
Moreover, for any $\wp\in (0,\nicefrac{6}{7})$ we have
\[
\lim\limits_{\epsilon \to 0}\sup\limits_{0\leq t\leq \mathcal{O}\left(\nicefrac{1}{\epsilon^{\wp}}\right)}\bb E \left[\| x^\epsilon(t) -(\varphi(t)+\sqrt{\epsilon}y(t))\|^2\right]=0.
\]
\end{lemma}
\begin{proof}
Recall that $H^{\epsilon}(t)=\bb E \left[\|h^{\epsilon}(t)\|^2\right]$, where 
\[h^\epsilon(t)=F(\varphi(t)+\sqrt{\epsilon}y(t)) - F(\varphi(t)) - \sqrt{\epsilon} DF(\varphi(t))y(t) \quad \textrm{ for any } t\geq 0.\]
Take any $\eta>0$ and $t>0$ and define the event
\[
A^{}_{\eta,\epsilon,t}=\left[\sup\limits_{0\leq s\leq t}\|y(s)\|\leq \frac{\eta}{\sqrt{\epsilon}}\right].
\]
By inequality \eqref{cota2} we have
\[
\bb E \left[\|h^{\epsilon}(t)\|^2 \mathds{1}_{A_{\eta,\epsilon,t}}\right]
\leq C^2_0(\eta, \|x_0\|)\epsilon^2 \bb E \left[\|y(t)\|^4\right]
\]
for any $t\geq 0$,
where the positive constant $C_0(\eta, \|x_0\|)$ depends on $\eta$ and $\|x_0\|$. By a similar argument using in inequality \eqref{zoa1} we deduce 
$\bb E \left[\|y(t)\|^4\right]\leq \frac{d(d+2)}{4\delta^2}$ for any $t\geq 0$.
Then 
\[
\bb E \left[\|h^{\epsilon}(t)\|^2 \mathds{1}_{A_{\eta,\epsilon,t}}\right]\leq C^2_0(\eta, \|x_0\|) \frac{d(d+2)}{4\delta^2}\epsilon^2 \quad \textrm{ for any } t\geq 0.\] 
On the other hand, recall the well-known inequality
$(x+y+z)^2\leq 4(x^2+y^2+z^2)$ for any $x,y,z \in \mathbb{R}$.
Then 
\[
\begin{split}
\bb E \left[\|h^{\epsilon}(t)\|^2 \mathds{1}_{A^c_{\eta,\epsilon,t}}\right]
\leq &
4\bb E\left[\|F(\varphi(t)+\sqrt{\epsilon}y(t))\|^2 \mathds{1}_{A^c_{\eta,\epsilon,t}} \right ]
+4\bb E\left[\|F(\varphi(t))\|^2\mathds{1}_{A^c_{\eta,\epsilon,t}}\right]
\\
&+ 4\epsilon \bb E\left[\|DF(\varphi(t))y(t)\|^2\mathds{1}_{A^c_{\eta,\epsilon,t}}\right]
\end{split}
\]
for any $t\geq 0$.
We will analyse the upper bound of the last inequality.
Since $\|\varphi(s)\|\leq \|x_0\|$ for any $s\geq 0$, then 
\[
\bb E\left[\|F(\varphi(t))\|^2\mathds{1}_{A^c_{\eta,\epsilon,t}}\right]
\leq C^2_1(\|x_0\|) \mathbb{P}\left(A^c_{\eta,\epsilon,t} \right),
\]
for any $t\geq 0$,
where $C_1(\|x_0\|)$ is a positive constant that only depends on $\|x_0\|$.
We also observe  that
\[
 \bb E\left[\|DF(\varphi(t))y(t)\|^2\mathds{1}_{A^c_{\eta,\epsilon,t}}\right]\leq C^2_2(\|x_0\|)\bb E\left[\|y(t)\|^2\mathds{1}_{A^c_{\eta,\epsilon,t}}\right]
\]
for any $t\geq 0$,
where $C_2(\|x_0\|)$ is a positive constant that only depends on $\|x_0\|$.
From the Cauchy--Schwarz inequality we get
\[
\begin{split}
\bb E\left[\|y(t)\|^2\mathds{1}_{A^c_{\eta,\epsilon,t}}\right]=&
\bb E\left[\left(\|y(t)\|^2\mathds{1}_{A^c_{\eta,\epsilon,t}}\right)\mathds{1}_{A^c_{\eta,\epsilon,t}}\right] \\
\leq &
\left(\bb E\left[\left(\|y(t)\|^4\mathds{1}_{A^c_{\eta,\epsilon,t}}\right)\right]\right)^{\nicefrac{1}{2}}
\left(\mathbb{P}\left(A^c_{\eta,\epsilon,t}\right)\right)^{\nicefrac{1}{2}} \\
\leq &
\left(\bb E\left[\left(\|y(t)\|^8\right)\right]\right)^{\nicefrac{1}{4}}
\left(\mathbb{P}\left(A^c_{\eta,\epsilon,t}\right)\right)^{\nicefrac{3}{4}}.
\end{split}
\]
for any $t\geq 0$.
Following similar computations as we did in \eqref{zoa1} or by item ii) of Proposition \ref{a16} we deduce 
\[\bb E\left[\left(\|y(t)\|^8\right)\right]\leq \frac{d(d+2)(d+4)(d+6)}{16\delta^4} \quad \textrm{ for any } t\geq 0.\]
Therefore
\[
\bb E\left[\|DF(\varphi(t))y(t)\|^2\mathds{1}_{A^c_{\eta,\epsilon,t}}\right]\leq C^2_3(\|x_0\|,\delta,d)\left(\mathbb{P}\left(A^c_{\eta,\epsilon,t}\right)\right)^{\nicefrac{3}{4}}
\]
for any $t\geq 0$,
where $C_3(\|x_0\|,\delta,d)$ is a positive constant that only depends on $\|x_0\|$, $\delta$ and $d$. 
Finally, we analise $\bb E\left[\|F(\varphi(t)+\sqrt{\epsilon}y(t))\|^2 \mathds{1}_{A^c_{\eta,\epsilon,t}} \right ]$.
From \eqref{C3} we have
\[
\begin{split}
\bb E  &\left[\|F(\varphi(t)+\sqrt{\epsilon}y(t))\|^2  \mathds{1}_{A^c_{\eta,\epsilon,t}} \right ]\leq 
c^2_0e^{4c_1\|x_0\|^2}\bb E\left[e^{4c_1\epsilon\|y(t)\|^2}\mathds{1}_{A^c_{\eta,\epsilon,t}}\right]
\end{split}
\] 
for any $t\geq 0$.
From the Cauchy-Schwarz inequality we deduce
\[
\begin{split}
\bb E \left[\left(e^{4c_1\epsilon\|y(t)\|^2}\mathds{1}_{A^c_{\eta,\epsilon,t}}\right)\mathds{1}_{A^c_{\eta,\epsilon,t}}\right]
\leq &
\left(\bb E\left[e^{8c_1\epsilon\|y(t)\|^2}\mathds{1}_{A^c_{\eta,\epsilon,t}}\right]\right)^{\nicefrac{1}{2}}
\left(\mathbb{P}\left(A^c_{\eta,\epsilon,t}\right)\right)^{\nicefrac{1}{2}}\\
 \leq & \left(\bb E\left[e^{16c_1\epsilon\|y(t)\|^2}\right]\right)^{\nicefrac{1}{4}}
\left(\mathbb{P}\left(A^c_{\eta,\epsilon,t}\right)\right)^{\nicefrac{3}{4}}
\end{split}
\]
for any $t\geq 0$.
From item iv) of Proposition \ref{a16}, for any 
$\epsilon\in \left(0,\frac{\delta}{32c_1}\right)$
 we have
\[
\bb E\left[e^{16c_1\epsilon\|y(t)\|^2}\right]\leq e^{16c_1d\epsilon t} \quad \textrm{ for any } t\geq 0.
\]
Therefore, 
\[
\begin{split}
\bb E\left[\|h^{\epsilon}(t)\|^2\right]
 \leq &  C^2_4(\eta,\|x_0\|,d,\delta)\epsilon^2+
4C^2_1(\|x_0\|)\mathbb{P}\left(A^{c}_{\eta,\epsilon,t}\right)\\
&\hspace{-2.0cm} +4\left(C^2_3(\|x_0\|,\delta,d)+e^{4c_1d\epsilon t}\right)\left(\mathbb{P}\left(A^{c}_{\eta,\epsilon,t}\right)\right)^{\nicefrac{3}{4}}\\
& \hspace{-2.5cm}\leq  C^2_4(\eta,\|x_0\|,d,\delta)\epsilon^2+
4\left(C^2_1(\|x_0\|)+C^2_3(\|x_0\|,\delta,d)+e^{4c_1d\epsilon t}\right)\left(\mathbb{P}\left(A^{c}_{\eta,\epsilon,t}\right)\right)^{\nicefrac{3}{4}},
\end{split}
\]
where $C^2_4(\eta,\|x_0\|,d,\delta)=C^2_0(\eta,\|x_0\|)\frac{d(d+2)}{4\delta^2}$. 
From item ii) of Lemma \ref{uap} we have
\[
\mathbb{P}\left(A^{c}_{\eta,\epsilon,t}\right)\leq 
\frac{2d\epsilon^2 t}{\delta(\eta^2-\epsilon d t)^2}\quad
\textrm{ for any } 0\leq t< \frac{\eta^2}{\epsilon d}.
\]
Notice that 
\[
\mathbb{P}\left(A^{c}_{\eta,\epsilon,t}\right)\leq 
\frac{4d\epsilon^2 t}{\delta \eta^2}\quad
\textrm{ for any } 0\leq t< \frac{\eta^2}{2\epsilon d}.
\]
Consequently,
\[
\begin{split}
\bb E\left[\|h^{\epsilon}(t)\|^2\right]
\leq  & C^2_4(\eta,\|x_0\|,d,\delta)\epsilon^2+\\
& 4\left(C^2_1(\|x_0\|)+C^2_3(\|x_0\|,\delta,d)+e^{2c_1\eta^2}\right)\left(
\frac{4d\epsilon^2 t}{\delta \eta^2}
\right)^{\nicefrac{3}{4}}
\end{split}
\]
for any $0\leq t< \frac{\eta^2}{2\epsilon d}$.
The second part follows immediately from inequality \eqref{pario}.
\end{proof}

In Lemma \ref{covmat}, we will prove that the linear non-homogeneous process $\{y(t):t\geq 0\}$ has a limiting, non-degenerate law which is Gaussian with mean vector  zero and covariance matrix $\Sigma$ which is the unique solution of the Lyapunov matrix equation \eqref{lyapv}.

\subsection{An $\nicefrac{\epsilon}{3}$ proof}\label{gera}
We approximate the process $\{x^{\epsilon}(t): t\geq 0\}$ by a linear non-homogeneous process $\{z^{\epsilon}(t):=\varphi(t)+\sqrt{\epsilon}y(t):t\geq 0\}$ in which we can carry out ``explicit'' computations.
Since we need to compare solutions of various Stochastic Differential Equations with different initial conditions, we introduce some notation.
Let
$\xi $ be a random variable in $\mathbb R^d$ and let $T>0$. Let $\{\varphi(t,\xi ):t \geq 0\}$ denote the solution of
\begin{equation*}
\left\{
\begin{array}{r@{\;=\;}l}
\ud {\varphi(t,\xi )} & -F(\varphi(t,\xi ))\ud t \quad \textrm{ for any }t\geq 0, \\
\varphi(0,\xi ) & \xi.
\end{array}
\right.
\end{equation*}
Let $\{y(t,\xi,T):t \geq 0\}$ be the solution of the stochastic differential equation
\begin{eqnarray*}\label{pois}
\left\{
\begin{array}{r@{\;=\;}l}
\ud {y(t,\xi,T)} & -DF(\varphi(t,\xi))y(t,\xi,T)\ud t+\ud B({t+T}) \textrm{ for any }t\geq 0,\\
y(0,\xi,T) & 0
\end{array}
\right.
\end{eqnarray*}
and define $\{z^\epsilon(t,\xi,T):t \geq 0\}$ as $z^\epsilon(t,\xi,T) := \varphi(t,\xi) + \sqrt \epsilon y(t,\xi,T)$ for any $t\geq 0$.

Let $c\in \mathbb{R}$.
In what follows, we will always take $T=t_{\mathrm{mix}}^\epsilon + c w^\epsilon>0$ for every $\epsilon>0$ small enough, so for simplicity, we will omit it from the notation.

Let $\delta_\epsilon>0$ such that $\delta_\epsilon=o(1)$. For $\epsilon\ll 1$ define
\begin{equation}\label{kjhg}
t_{\mathrm{shift}}^\epsilon:=t_{\mathrm{mix}}^\epsilon-\delta_\epsilon>0.
\end{equation}

The following lemma is the key of the proof. 
Roughly speaking,
from Lemma \ref{falta} we see that the processes $\{x^{\epsilon}(t):t\geq 0\}$
and $\{z^{\epsilon}(t):t\geq 0\}$ are close enough 
for times of order
$\mathcal{O}(\ln(\nicefrac{1}{\epsilon}))$. Therefore we 
can shift the processes for a small time $\delta_\epsilon$ and then we coupled the remainder differences in a small time interval $[0,\delta_\epsilon]$. Since $\{z^{\epsilon}(t):t\geq 0\}$ is linear, then thermalisation (window cut-off) will be concluded from it.

\begin{lemma}\label{goodine}
For any $c\in \mathbb{R}$ and $\epsilon\ll 1 $ we have
\begin{equation}\label{buen}
\begin{split}
&\left|\ud_{\mathrm{TV}}(x^{\epsilon}(t_{\mathrm{mix}}^\epsilon+cw^\epsilon,x_0),\mu^{\epsilon})-
\ud_{\mathrm{TV}}(z^{\epsilon}(\delta_\epsilon,z^{\epsilon}(t_{\mathrm{shift}}^\epsilon+cw^\epsilon,x_0)),\mathcal{G}(0,\epsilon\Sigma))\right|
 \leq \\
& \ud_{\mathrm{TV}}(x^{\epsilon}(\delta_\epsilon,x^{\epsilon}(t_{\mathrm{shift}}^\epsilon+cw^\epsilon,x_0)),
z^{\epsilon}(\delta_\epsilon,x^{\epsilon}(t_{\mathrm{shift}}^\epsilon+cw^\epsilon,x_0)))+\\
& \ud_{\mathrm{TV}}(z^{\epsilon}(\delta_\epsilon,x^{\epsilon}(t_{\mathrm{shift}}^\epsilon+cw^\epsilon,x_0)),
z^{\epsilon}(\delta_\epsilon,z^{\epsilon}(t_{\mathrm{shift}}^\epsilon+cw^\epsilon,x_0)))+
\ud_{\mathrm{TV}}(\mathcal{G}(0,\epsilon\Sigma),\mu^\epsilon).
\end{split}
\end{equation}
\end{lemma}
\begin{proof}
Notice that
\begin{align*}
\ud_{\mathrm{TV}}&(x^{\epsilon}(t_{\mathrm{shift}}^\epsilon+cw^\epsilon+\delta_\epsilon,x_0),\mu^{\epsilon})=
\ud_{\mathrm{TV}}(x^{\epsilon}(\delta_\epsilon,x^{\epsilon}(t_{\mathrm{shift}}^\epsilon+cw^\epsilon,x_0)),\mu^{\epsilon}) \leq \\
& \ud_{\mathrm{TV}}(x^{\epsilon}(\delta_\epsilon,x^{\epsilon}(t_{\mathrm{shift}}^\epsilon+cw^\epsilon,x_0)),
z^{\epsilon}(\delta_\epsilon,x^{\epsilon}(t_{\mathrm{shift}}^\epsilon+cw^\epsilon,x_0)))+\\
& \ud_{\mathrm{TV}}(z^{\epsilon}(\delta_\epsilon,x^{\epsilon}(t_{\mathrm{shift}}^\epsilon+cw^\epsilon,x_0)),
z^{\epsilon}(\delta_\epsilon,z^{\epsilon}(t_{\mathrm{shift}}^\epsilon+cw^\epsilon,x_0)))+\\
& \ud_{\mathrm{TV}}(z^{\epsilon}(\delta_\epsilon,z^{\epsilon}(t_{\mathrm{shift}}^\epsilon+cw^\epsilon,x_0)),\mathcal{G}(0,\epsilon\Sigma))
+\ud_{\mathrm{TV}}(\mathcal{G}(0,\epsilon\Sigma),\mu^\epsilon).
\end{align*}
On the other hand, 
\begin{align*}
&\ud_{\mathrm{TV}}(z^{\epsilon}(\delta_\epsilon,z^{\epsilon}(t_{\mathrm{shift}}^\epsilon+cw^\epsilon,x_0)),\mathcal{G}(0,\epsilon\Sigma)) \leq \\
& \ud_{\mathrm{TV}}(z^{\epsilon}(\delta_\epsilon,z^{\epsilon}(t_{\mathrm{shift}}^\epsilon+cw^\epsilon,x_0)),
z^{\epsilon}(\delta_\epsilon,x^{\epsilon}(t_{\mathrm{shift}}^\epsilon+cw^\epsilon,x_0)))+\\
& \ud_{\mathrm{TV}}(z^{\epsilon}(\delta_\epsilon,x^{\epsilon}(t_{\mathrm{shift}}^\epsilon+cw^\epsilon,x_0)),
x^{\epsilon}(\delta_\epsilon,x^{\epsilon}(t_{\mathrm{shift}}^\epsilon+cw^\epsilon,x_0)))+\\
& \ud_{\mathrm{TV}}(x^{\epsilon}(\delta_\epsilon,x^{\epsilon}(t_{\mathrm{shift}}^\epsilon+cw^\epsilon,x_0)),\mu^{\epsilon})
+\ud_{\mathrm{TV}}(\mu^\epsilon,\mathcal{G}(0,\epsilon\Sigma)).
\end{align*}
Gluing both inequalities we deduce

\begin{align*}\label{ine5}
&\left|\ud_{\mathrm{TV}}(x^{\epsilon}(t_{\mathrm{shift}}^\epsilon+cw^\epsilon+\delta_\epsilon,x_0),\mu^{\epsilon})-
\ud_{\mathrm{TV}}(z^{\epsilon}(\delta_\epsilon,z^{\epsilon}(t_{\mathrm{shift}}^\epsilon+cw^\epsilon,x_0)),\mathcal{G}(0,\epsilon\Sigma))\right|
 \leq \nonumber\\
& \ud_{\mathrm{TV}}(x^{\epsilon}(\delta_\epsilon,x^{\epsilon}(t_{\mathrm{shift}}^\epsilon+cw^\epsilon,x_0)),
z^{\epsilon}(\delta_\epsilon,x^{\epsilon}(t_{\mathrm{shift}}^\epsilon+cw^\epsilon,x_0)))+\\
& \ud_{\mathrm{TV}}(z^{\epsilon}(\delta_\epsilon,x^{\epsilon}(t_{\mathrm{shift}}^\epsilon+cw^\epsilon,x_0)),
z^{\epsilon}(\delta_\epsilon,z^{\epsilon}(t_{\mathrm{shift}}^\epsilon+cw^\epsilon,x_0)))+
\ud_{\mathrm{TV}}(\mathcal{G}(0,\epsilon\Sigma),\mu^\epsilon).\nonumber
\end{align*}
\end{proof}
In what follows, we will prove that the upper bound of inequality \eqref{buen} is negligible as $\epsilon\rightarrow 0$.
\subsubsection{Short-time coupling}
A natural question arising is how to obtain explicit  ``good" bounds
for the total variation distance between $x^{\epsilon}(t)$ and $z^{\epsilon}(t)$.
 Using the celebrated 
Cameron-Martin-Girsanov Theorem, a coupling on the path space can be done and it is possible to
establish bounds on the total variation distance using the Pinsker inequality of such diffusions. This method only provides a coupling over short time intervals. For more details see \cite{EZ},
\cite{BJ}, \cite{KALS} and the references therein.
On the other hand, ``explicit" bounds for the total variation distance between transition probabilities of diffusions with different drifts are
derived using analytic arguments. 
This approach also works for the stationary measures of the diffusions.
For further details see \cite{BRS} and the references therein.

In order to avoid homogenisation arguments for $F$, 
we use the Hellinger approach
developed in \cite{KALS} for obtain an upper bound for the total variation distance between the non-linear model $x^{\epsilon}(t)$  with the linear non-homogeneous model $z^{\epsilon}(t)$ in a short time interval. That upper bound is enough for our purposes. As we can notice in Theorem $5.1$ in \cite{KALS}, we need to carry out second-moment estimates of the distance between the vector fields associated to the diffusions $\{x^{\epsilon}(t):t\geq 0\}$ and $\{z^{\epsilon}(t):t\geq 0\}$, respectively. It is exactly the estimate that we did in Lemma \ref{falta}. 

\begin{proposition}\label{a18}
Assume that 
 \eqref{C2} and   \eqref{C3} hold.
Let $\delta_\epsilon>0$ such that $\delta_\epsilon=o(1)$. Then for any $c\in \mathbb{R}$
\[
\lim\limits_{\epsilon\rightarrow 0}
\ud_{\mathrm{TV}}(x^{\epsilon}(\delta_\epsilon,x^{\epsilon}(t_{\mathrm{shift}}^\epsilon+cw^\epsilon,x_0)),
z^{\epsilon}(\delta_\epsilon,x^{\epsilon}(t_{\mathrm{shift}}^\epsilon+cw^\epsilon,x_0)))=0,
\]
where $t_{\mathrm{shift}}^\epsilon$ is given by \eqref{kjhg}.
\end{proposition}

\begin{proof}
Let $T^\epsilon=t_{\mathrm{shift}}^\epsilon + c w^\epsilon>0$ for $\epsilon\ll 1$.  
Notice that
\[
\begin{split}
\ud_{\mathrm{TV}}&(x^{\epsilon}(\delta_\epsilon,x^{\epsilon}(T^{\epsilon},x_0)),z^{\epsilon}(\delta_\epsilon,x^{\epsilon}(T^{\epsilon},x_0)))\leq \\
&\int\limits_{\mathbb{R}^d}
\ud_{\mathrm{TV}}(x^{\epsilon}(\delta_\epsilon,u),z^{\epsilon}(\delta_\epsilon,{u}))
\mathbb{P}(x^{\epsilon}(T^{\epsilon},x_0)\in \ud u).
\end{split}
\]
For short, denote by $\mathbb{P}^\epsilon(\ud u)$  the probability measure
$\mathbb{P}(x^{\epsilon}(T^{\epsilon},x_0)\in \ud u)$.
Let $K$ be a positive constant. Then
\begin{equation}\label{aci1}
\begin{split}
\ud_{\mathrm{TV}}&(x^{\epsilon}(\delta_\epsilon,x^{\epsilon}(T^{\epsilon},x_0)),z^{\epsilon}(\delta_\epsilon,x^{\epsilon}(T^{\epsilon},x_0)))\leq \\
&\int\limits_{\|u\|\leq K}
\ud_{\mathrm{TV}}(x^{\epsilon}(\delta_\epsilon,u),z^{\epsilon}(\delta_\epsilon,{u}))\mathbb{P}^{\epsilon}(\ud u)
+\mathbb{P}\left(\|x^{\epsilon}(T^{\epsilon},x_0)\|>K\right).
\end{split}
\end{equation}
Now, we prove that the upper bound of \eqref{aci1} is negligible as $\epsilon\rightarrow 0$. From the Markov inequality we get
\[
\mathbb{P}\left(\|x^{\epsilon}(T^{\epsilon},x_0)\|>K\right)\leq 
\frac{\mathbb{E}\left[
\|x^{\epsilon}(T^{\epsilon},x_0)\|^2\right]}{K^2}.
\]
Recall the well--known inequality $(x+y)^2\leq 2(x^2+y^2)$ for any $x,y\in \mathbb{R}$. Then
\[
\begin{split}
\mathbb{E}\left[
\|x^{\epsilon}(T^{\epsilon},x_0)\|^2\right]
\leq & 2\mathbb{E}\left[
\|x^{\epsilon}(T^{\epsilon},x_0)-\varphi(T^{\epsilon},x_0)\|^2\right]\\
&+2\|\varphi(T^{\epsilon},x_0)\|^2.
\end{split}
\]
From inequality \eqref{zeroth} and inequality \eqref{mono} we have 
\[
\mathbb{E}\left[
\|x^{\epsilon}(T^{\epsilon},x_0)\|^2\right]
\leq \frac{\epsilon d}{\delta}+2 e^{-2\delta T^{\epsilon}}\|x_0\|^2,
\]
which allows to deduce 
\begin{equation}\label{pio1}
\lim\limits_{\epsilon \rightarrow 0}\mathbb{P}\left(\|x^{\epsilon}(T^{\epsilon},x_0)\|>{K}\right)=0.
\end{equation}
Now, we analyse $\int\limits_{\|u\|\leq K}
\ud_{\mathrm{TV}}(x^{\epsilon}(\delta_\epsilon,u),z^{\epsilon}(\delta_\epsilon,{u}))\mathbb{P}^{\epsilon}(\ud u)$.
From the Theorem $5.1$ in \cite{KALS} we obtain
\[
\begin{split}
&\int\limits_{\|u\|\leq K}
\ud_{\mathrm{TV}}(x^{\epsilon}(\delta_\epsilon,u),z^{\epsilon}(\delta_\epsilon,{u}))\mathbb{P}^{\epsilon}(\ud u)\leq 
\frac{1}{\epsilon}\int\limits_{\|u\|\leq K}
\int\limits_{0}^{\delta_\epsilon}\bb E
\left[
\|I^\epsilon(s,u)
\|^2
\right]\ud s \mathbb{P}^{\epsilon}(\ud u),
\end{split}
\]
where 
\[
I^\epsilon(s,u):=F(x^\epsilon(s,u)
-DF(\varphi(s,u))z^{\epsilon}(s,u)+
DF(\varphi(s,u))\varphi(s,u)-F(\varphi(s,u))
\]
for any $s\geq 0$ and $u\in \mathbb{R}^d$.
Following the same argument using in Lemma \ref{falta},
for any $u\in \mathbb{R}^d$ such that $\|u\|\leq K$ and 
$0\leq s\leq \delta_\epsilon$ we deduce 
\[
\bb E \left[\|I^\epsilon(s,u)\|^2\right]\leq \frac{1}{\delta^2}\left(C(K,d,\delta)\epsilon^2 \delta_\epsilon+\tilde{C}(K,d,\delta)\epsilon^{\nicefrac{3}{2}}\delta_\epsilon^{\nicefrac{7}{4}}\right),
\]
where $C(K,d,\delta)$ and $\tilde{C}(K,d,\delta)$ only depend on $K$, $d$ and $\delta$. Therefore, 
\begin{equation}\label{pio2}
\lim\limits_{\epsilon \rightarrow 0}
\int\limits_{\|u\|\leq K}
\ud_{\mathrm{TV}}(x^{\epsilon}(\delta_\epsilon,u),z^{\epsilon}(\delta_\epsilon,{u}))\mathbb{P}^{\epsilon}(\ud u)=0.
\end{equation}
Inequality \eqref{aci1} together with relations \eqref{pio1} and \eqref{pio2} allow us to deduce  the desired result.
\end{proof}

\subsubsection{Linear non-homogeneous coupling}
In this part, we couple two non-homogeneous solutions $z^{\epsilon}(t,x)$ and $z^{\epsilon}(t,y)$ for short time $t\ll 1$ and initials conditions $x$ and $y$ such that $\|x-y\|$ small enough.  

\begin{proposition}\label{a19}
Assume that 
 \eqref{C2} and   \eqref{C3} hold.
Let $\delta_\epsilon=\epsilon^{\theta}$ for some 
$\theta\in (0,\nicefrac{1}{2})$.
Then for any $c\in \mathbb{R}$
\[
\lim\limits_{\epsilon\rightarrow 0}\ud_{\mathrm{TV}}(z^{\epsilon}(\delta_\epsilon,x^{\epsilon}(t_{\mathrm{shift}}^\epsilon+cw^\epsilon,x_0)),z^{\epsilon}(\delta_\epsilon,z^{\epsilon}(t_{\mathrm{shift}}^\epsilon+cw^\epsilon,x_0)))
=0,
\]
where $t_{\mathrm{shift}}^\epsilon$ is given by \eqref{kjhg}.
\end{proposition}
\begin{proof}
Let $T^\epsilon=t_{\mathrm{shift}}^\epsilon + c w^\epsilon>0$ for $\epsilon\ll 1$.  
Notice that
\[
\begin{split}
&\ud_{\mathrm{TV}}(z^{\epsilon}(\delta_\epsilon,x^{\epsilon}(T^\epsilon,x_0)),z^{\epsilon}(\delta_\epsilon,z^{\epsilon}(T^\epsilon,x_0)))\leq \\
&\int\limits_{\mathbb{R}^d\times \mathbb{R}^d}
\ud_{\mathrm{TV}}(z^{\epsilon}(\delta_\epsilon,u),z^{\epsilon}(\delta_\epsilon,\tilde{u}))\mathbb{P}(x^{\epsilon}(T^\epsilon,x_0)\in \ud u, z^{\epsilon}(T^\epsilon,x_0)\in \ud \tilde{u}),
\end{split}
\]
For short, we denote by $\mathbb{P}^\epsilon(\ud u, \ud \tilde{u})$ for the coupling 
\[\mathbb{P}(x^{\epsilon}(T^\epsilon,x_0)\in \ud u, z^{\epsilon}(T^\epsilon,x_0)\in \ud \tilde{u}).\]
Let $K$ and $\tilde{K}$ any positive constants. Then
\begin{equation}\label{cid1}
\begin{split}
\ud_{\mathrm{TV}}&(z^{\epsilon}(\delta_\epsilon,x^{\epsilon}(T^\epsilon,x_0)),z^{\epsilon}(\delta_\epsilon,z^{\epsilon}(T^\epsilon,x_0)))\leq \\
&\int\limits_{\|u\|\leq K, \|\tilde{u}\|\leq \tilde{K}}
\ud_{\mathrm{TV}}(z^{\epsilon}(\delta_\epsilon,u),z^{\epsilon}(\delta_\epsilon,\tilde{u}))\mathbb{P}^{\epsilon}(\ud u, \ud \tilde{u})+
\\
&\mathbb{P}\left(\|x^{\epsilon}(T^{\epsilon},x_0)\|>K\right)+
\mathbb{P}\left(\|z^{\epsilon}(T^{\epsilon},x_0)\|>\tilde{K}\right).
\end{split}
\end{equation}
Now, we prove that the upper bound of \eqref{cid1} is negligible as $\epsilon\rightarrow 0$. From relation \eqref{pio1} we have
\[
\lim\limits_{\epsilon \to 0}\mathbb{P}\left(\|x^{\epsilon}(T^{\epsilon},x_0)\|>K\right)=0.
\]
Similar ideas as in the proof of relation \eqref{pio1} and item ii) of Proposition \ref{a16} yield
\[
\lim\limits_{\epsilon \to 0}\mathbb{P}\left(\|z^{\epsilon}(T^{\epsilon},x_0)\|>\tilde{K}\right)=0.
\]
Since the stochastic differential equation associated to 
$\{y(t,u,T^\epsilon):t\geq 0\}$ is linear then the variation of parameters formula allows us to deduce that
\[
z^{\epsilon}(\delta_\epsilon,u)=\Phi(\delta_\epsilon)u+
\sqrt{\epsilon}(\Phi(\delta_\epsilon))
\int\limits_{0}^{\delta_\epsilon}(\Phi(s))^{-1}\ud (B(T^\epsilon+s)- B(T^\epsilon))
\]
for any $u\in \mathbb{R}^d$,
where
$\{\Phi(t): t\geq 0\}$ is the solution of the matrix differential equation:
\begin{equation*}
\left\{
\begin{array}{r@{\;=\;}l}
\tfrac{\ud}{\ud t} \Phi(t)&-\Phi(t) DF(\varphi(t+T^\epsilon)) \quad  
\textrm{  for } t\geq 0,\\
\Phi(0) & \mathrm{I}_d.
\end{array}
\right.
\end{equation*}
Observe that for any $\upsilon\in \mathbb{R}^d$, 
$z^{\epsilon}(\delta_\epsilon,\upsilon)$ has Gaussian distribution with mean vector $\varphi(\delta_\epsilon,\upsilon)$ and covariance matrix $\epsilon \Sigma(\delta_\epsilon)$, where $\Sigma(\delta_\epsilon)$ is the covariance matrix of the random vector 
\[
(\Phi(\delta_\epsilon))
\int\limits_{0}^{\delta_\epsilon}(\Phi(s))^{-1}\ud (B(T^\epsilon+s)- B(T^\epsilon))
\]
which does not depend on $\upsilon$.
Moreover, using the It\^o formula we deduce 
\[
\Sigma(\delta_\epsilon)=\Phi(\delta_\epsilon)\int\limits_{0}^{\delta_\epsilon}(\Phi(s))^{-1}\left((\Phi(s))^{-1}\right)^*\ud s
(\Phi(\delta_\epsilon))^*
\]
From the Lebesgue Differentiation Theorem we obtain
\begin{equation*}\label{ldt}
\lim\limits_{\epsilon \to 0}\frac{\Sigma(\delta_\epsilon)}{\delta_\epsilon}=\mathrm{I_d}.
\end{equation*} 
The latter allows us to deduce  that
$
\|(\Sigma(\delta_\epsilon))^
{-\nicefrac{1}{2}}\|\leq 
(\delta_\epsilon)^{-\nicefrac{1}{2}}
C(d),
$
where $C(d)$ is an absolute positive constant that only depends on $d$.
From item iii) of Lemma \ref{a1} and Lemma \ref{a2} we have
\[
\ud_{\mathrm{TV}}(z^{\epsilon}(\delta_\epsilon,u),z^{\epsilon}(\delta_\epsilon,\tilde{u}))\leq \frac{1}{\sqrt{2\pi\epsilon}}
\|(\Sigma(\delta_\epsilon))^
{-\nicefrac{1}{2}}\Phi(\delta_\epsilon)\left(u-\tilde{u}\right)\|
\]
for any $u,\tilde{u}\in \mathbb{R}^d$.
Then
\[
\ud_{\mathrm{TV}}(z^{\epsilon}(\delta_\epsilon,u),z^{\epsilon}(\delta_\epsilon,\tilde{u}))\leq \frac{C_1(d)}{\sqrt{\epsilon \delta_\epsilon}}
\|u-\tilde{u}\|\quad \textrm{ for any } u,\tilde{u}\in \mathbb{R}^d,
\]
where
$C_1(d)$ is an absolute positive constant that only depends on $d$. Therefore,
\[ 
\begin{split}
\int\limits_{\|u\|\leq K,\|\tilde{u}\|\leq \tilde{K}}
\ud_{\mathrm{TV}}(z^{\epsilon}(\delta_\epsilon,u),z^{\epsilon}(\delta_\epsilon,\tilde{u}))\mathbb{P}^{\epsilon}(\ud u, \ud \tilde{u})\leq &\\
&\hspace{-8cm}\frac{C_1(d)}{\sqrt{\epsilon \delta_\epsilon}} \mathbb{E}\left[\|x^{\epsilon}(T^\epsilon,x_0)-z^{\epsilon}(T^\epsilon,x_0) \|\right]\leq 
\frac{C_1(d)}{\sqrt{\epsilon \delta_\epsilon}} \left(\mathbb{E}\left[\|x^{\epsilon}(T^\epsilon,x_0)-z^{\epsilon}(T^\epsilon,x_0) \|^2\right]\right)^{\nicefrac{1}{2}}.
\end{split}
\]
From Lemma \ref{falta} we deduce 
\[
\lim\limits_{\epsilon \to 0}
\int\limits_{\|u\|\leq K,\|\tilde{u}\|\leq \tilde{K}}
\ud_{\mathrm{TV}}(z^{\epsilon}(\delta_\epsilon,u),z^{\epsilon}(\delta_\epsilon,\tilde{u}))\mathbb{P}^{\epsilon}(\ud u, \ud \tilde{u})=0.
\]
Putting all pieces together, we get the statement.
\end{proof}

\subsubsection{Window Cut-off}
Remind that $z^{\epsilon}(t)=\varphi(t)+\sqrt{\epsilon}y(t)$, $t\geq 0$, where $\{y(t): t\geq 0\}$ satisfies the linear non-homogeneous stochastic differential equation:
\begin{equation*}
\label{linearhomo}
\left\{
\begin{array}{r@{\;=\;}l}
\ud y(t) & -DF(\varphi(t)) y(t) \ud t +\ud B(t) \quad \textrm{ for }t\geq 0,\\
y(0) & 0.
\end{array}
\right.
\end{equation*}
Therefore, for any $t>0$, $z^{\epsilon}(t)$ has Gaussian distribution with zero mean vector $\varphi(t)$ and covariance matrix $\epsilon\Sigma (t)$, where 
$\{\Sigma(t): t\geq 0\}$ 
is the the solution to the deterministic matrix differential equation:
\begin{equation*}
\left\{
\begin{array}{r@{\;=\;}l}
\frac{\ud }{\ud t}\Sigma(t) & -DF(\varphi(t))\Sigma(t)-\Sigma(t)(DF(\varphi(t)))^*+\mathrm{I}_d \quad \textrm{ for }t\geq 0,\\
\Sigma(0) & 0.
\end{array}
\right.
\end{equation*}

Under \eqref{C2}, we can prove that $\varphi(t)\rightarrow 0$  and
$\Sigma(t)\rightarrow \Sigma$ as $t\rightarrow +\infty$, where $\Sigma$ is a symmetric and positive definite matrix (See Lemma \ref{covmat}).
Therefore, $z^{\epsilon}(t)$ converges in distribution to a random vector $z^{\epsilon}(\infty)$ as  $t\rightarrow +\infty$, where $z^{\epsilon}(\infty)$ has Gaussian law with zero mean vector  and covariance matrix $\epsilon\Sigma$.
Using item iii) of Lemma \ref{a1}, Lemma \ref{a4}, Lemma \ref{a6}  the convergence can be easily improved to be in total variation distance.
Bearing all this in mind, we can measure how abrupt is the convergence to its equilibrium. Define
\[
\bar{D}^{\epsilon}(t):=\ud_{\mathrm{TV}}(z^{\epsilon}(t),z^{\epsilon}(\infty))=\ud_{\mathrm{TV}}(\mathcal{G}(\varphi(t), \epsilon\Sigma(t)),\mathcal{G}(0, \epsilon\Sigma))
\]
for any $t>0$.

\begin{proposition}\label{prop2}
Assume that 
 \eqref{C2} holds.
Let $\delta_\epsilon\geq 0$ such that $\delta_\epsilon=o\left(1\right)$.
For any $c\in \mathbb{R}$ we have
\[
\lim\limits_{\epsilon\rightarrow 0}\Big |\bar{D}^{\epsilon}(t_{\mathrm{shift}}^\epsilon+\delta_\epsilon+cw^{\epsilon})-{D}^{\epsilon}(t_{\mathrm{shift}}^\epsilon+\delta_\epsilon+cw^{\epsilon}) \Big |=0,
\]
where $t_{\mathrm{shift}}^\epsilon$ is given by \eqref{kjhg}, 
\[{D}^{\epsilon}(t):=\ud_{\mathrm{TV}}\left(\mathcal{G}\left(\frac{(t-\tau)^{\ell-1}}{e^{\lambda (t-\tau)}\sqrt{\epsilon}}\Sigma^{-\nicefrac{1}{2}}\sum\limits_{k=1}^{m} e^{i \theta_k (t-\tau)} v_k,\mathrm{I}_d \right),\mathcal{G}(0,\mathrm{I}_d)\right)\]
for any $t\geq \tau$, with $\lambda$, $\ell$, $\tau$, $\theta_1,\ldots, \theta_m \in [0,2\pi)$, $v_1,\ldots,v_m$ are the constants and vectors  associated to $x_0$
in {\em{Lemma \ref{asymp}}},
and the matrix $\Sigma$ is the unique solution of the matrix Lyapunov equation:
\begin{equation*}
\label{lyapv2}
DF(0)X+X (DF(0))^*=\mathrm{I}_d.
\end{equation*}
\end{proposition}

\begin{proof}
Let $t>0$. From the triangle inequality and item ii), item iii) of Lemma \ref{a1}  we obtain
\begin{align*}
\bar{D}^{\epsilon}(t)\leq &
\ud_{\mathrm{TV}}(\mathcal{G}(\varphi(t), \epsilon\Sigma(t)),\mathcal{G}(\varphi(t), \epsilon\Sigma))+
\ud_{\mathrm{TV}}(\mathcal{G}(\varphi(t), \epsilon\Sigma),\mathcal{G}(0, \epsilon\Sigma))\leq \\
&\ud_{\mathrm{TV}}(\mathcal{G}(0,\Sigma(t)),\mathcal{G}(0,\Sigma))+
\ud_{\mathrm{TV}}\left(\mathcal{G}\left(\frac{1}{\sqrt{\epsilon}} \varphi(t),\Sigma\right),\mathcal{G}(0,\Sigma)\right).
\end{align*}
A similar argument allows us to deduce that
\begin{equation}\label{fg1}
\Big |\bar{D}^{\epsilon}(t)-\ud_{\mathrm{TV}}\left(\mathcal{G}\left(\frac{1}{\sqrt{\epsilon}} \varphi(t),\Sigma\right),\mathcal{G}(0,\Sigma)\right) \Big|\leq
\ud_{\mathrm{TV}}(\mathcal{G}(0,\Sigma(t)),\mathcal{G}(0,\Sigma)).
\end{equation}
Using Lemma \ref{a6} we get
\[\lim\limits_{t\rightarrow +\infty}\ud_{\mathrm{TV}}(\mathcal{G}(0,\Sigma(t)),\mathcal{G}(0,\Sigma))=0.\]
Therefore, the cut-off phenomenon can be deduced from the distance
\begin{align*}
\hat{D}^{\epsilon}(t)&:=\ud_{\mathrm{TV}}\left(\mathcal{G}\left(\frac{1}{\sqrt{\epsilon}} \varphi(t),\Sigma\right),\mathcal{G}(0,\Sigma)\right)\\
&=\ud_{\mathrm{TV}}\left(\mathcal{G}\left(\Sigma^{-\nicefrac{1}{2}}\frac{1}{\sqrt{\epsilon}}\varphi(t),\mathrm{I}_d\right),\mathcal{G}(0,\mathrm{I}_d)\right)
\end{align*}
for any $t>0$,
where the last equality follows from item iii) of Lemma \ref{a1}.
Using the constants  and vectors associated to $x_0$ in Lemma \ref{asymp}, for any $t>\tau$ define
 \[
{D}^{\epsilon}(t):=\ud_{\mathrm{TV}}\left(\mathcal{G}\left(\Sigma^{-\nicefrac{1}{2}}\frac{(t-\tau)^{\ell-1}}{e^{\lambda (t-\tau)}\sqrt{\epsilon}}\sum\limits_{k=1}^{m} e^{i \theta_k (t-\tau)} v_k,\mathrm{I}_d\right),\mathcal{G}(0,\mathrm{I}_d)\right)
\]
and
\[
\textrm{R}^{\epsilon}(t):=
\ud_{\mathrm{TV}}\left(\mathcal{G}\left(\Sigma^{-\nicefrac{1}{2}}\frac{1}{\sqrt{\epsilon}}\varphi(t),\mathrm{I}_d\right),\mathcal{G}\left(\Sigma^{-\nicefrac{1}{2}}\frac{(t-\tau)^{\ell-1}}{e^{\lambda (t-\tau)}\sqrt{\epsilon}}\sum\limits_{k=1}^{m} e^{i \theta_k (t-\tau)} v_k,\mathrm{I}_d\right)\right).
\]
From item ii) of Lemma \ref{a1} we 
deduce that 
\[
\textrm{R}^{\epsilon}(t)=
\ud_{\mathrm{TV}}\left(\mathcal{G}\left(\frac{1}{\sqrt{\epsilon}}\Sigma^{-\nicefrac{1}{2}}\left(\varphi(t)-\frac{(t-\tau)^{\ell-1}}{e^{\lambda (t-\tau)}}\sum\limits_{k=1}^{m} e^{i \theta_k (t-\tau)} v_k\right),\mathrm{I}_d\right),\mathcal{G}\left(0,\mathrm{I}_d\right)\right).
\]
A similar argument  for obtain inequality \eqref{fg1} allows to  show  that
\begin{equation}\label{fg2}
\Big|\hat{D}^{\epsilon}(t)- {D}^{\epsilon}(t)\Big |\leq \textrm{R}^{}(t) \quad \textrm{ for any }
t>\tau.
\end{equation}
From inequalities \eqref{fg1} and\eqref{fg2} we obtain
\[\Big|{D}^{\epsilon}(t)- \bar{D}^{\epsilon}(t)\Big |\leq \textrm{R}^{\epsilon}(t)+\ud_{\mathrm{TV}}(\mathcal{G}(0,\Sigma(t)),\mathcal{G}(0,\Sigma))
\]
for any $t>\tau$.
Straightforward computations led us to
\begin{equation}\label{hju}
\lim\limits_{\epsilon \to 0}\frac{e^{-\lambda (t_{\mathrm{shift}}^\epsilon+\delta_\epsilon+cw^{\epsilon}-\tau)}(t_{\mathrm{shift}}^\epsilon+\delta_\epsilon+cw^{\epsilon}-\tau)^{\ell-1}}{\sqrt{\epsilon}}=(2\lambda)^{1-\ell}e^{-c}
\end{equation}
for any $c\in \mathbb{R}$.
Therefore, Lemma \ref{asymp} together relation \eqref{hju}    
 allow to deduce that
\[
\lim\limits_{\epsilon\rightarrow 0}\textrm{R}^{\epsilon}(t_{\mathrm{shift}}^\epsilon+\delta_\epsilon+cw^{\epsilon})=0
\]
for any $c\in \mathbb{R}$. Consequently, we obtain the statement.
\end{proof}

\subsubsection{The invariant measure}
In this section, we prove that the invariant measure $\mu^{\epsilon}$ of the evolution \eqref{SEDO} is well approximated in total variation distance by a Gaussian distribution with zero mean vector and covariance matrix $\epsilon\Sigma$,
where $\Sigma$ is the unique solution of the matrix Lyapunov equation:
\begin{equation*}
DF(0)X+X (DF(0))^*=\mathrm{I}_d.
\end{equation*}

\begin{proposition}\label{a20}
Assume that 
 \eqref{C2} and   \eqref{C3} hold. Then
\[
\lim\limits_{\epsilon\rightarrow 0}\ud_{\mathrm{TV}}(\mathcal{G}(0,\epsilon \Sigma),\mu^{\epsilon})=0.
\]
\end{proposition}
\begin{proof}
Recall that $z^{\epsilon}(t)=\varphi(t)+\sqrt{\epsilon}y(t)$ for any $t\geq 0$.
Note that for any $s,t\geq 0$ and $x\in \mathbb{R}^d$ we have
\begin{equation}\label{plo}
\begin{split}
\ud_{\mathrm{TV}}(\mathcal{G}(0,\epsilon \Sigma),\mu^{\epsilon})\leq & \ud_{\mathrm{TV}}(\mathcal{G}(0,\epsilon \Sigma),z^{\epsilon}(s+t,x))+\\
&\ud_{\mathrm{TV}}(z^{\epsilon}(s+t,x),x^{\epsilon}(s+t,x))+
\ud_{\mathrm{TV}}(x^{\epsilon}(s+t,x),\mu^{\epsilon}).
\end{split}
\end{equation}
Observe that
\begin{align*}
\ud_{\mathrm{TV}}(\mathcal{G}(0,\epsilon \Sigma),z^{\epsilon}(s+t,x))=\ud_{\mathrm{TV}}(\mathcal{G}(0,\epsilon \Sigma),\mathcal{G}(\varphi(s+t,x),\epsilon \Sigma(s+t))),
\end{align*}
where $\Sigma(t)$ is the covariance matrix of $y(t)$. Therefore, using the triangle inequality together with 
item ii) and item iii) of Lemma \ref{a1}
we obtain
\begin{equation}\label{rty}
\begin{split}
\ud_{\mathrm{TV}}(\mathcal{G}(0,\epsilon \Sigma),z^{\epsilon}(s+t,x))\leq &
\ud_{\mathrm{TV}}(\mathcal{G}(0, \Sigma),\mathcal{G}(0,\Sigma(s+t)))+\\
&
\ud_{\mathrm{TV}}\left(\mathcal{G}\left(\frac{\varphi(s+t,x)}{\sqrt{\epsilon}}, \Sigma\right),\mathcal{G}(0,\Sigma)\right).
\end{split}
\end{equation}
Let $s^{\epsilon}\ll
\epsilon^{\nicefrac{1}{2019}}$ and $t_{\mathrm{mix}}^\epsilon \ll t^{\epsilon}:=\frac{1}{\epsilon^{\nicefrac{1}{8}}}$.
By Lemma \ref{a6} and Lemma \ref{asympgeral} we obtain
\[
\lim\limits_{\epsilon\rightarrow 0}\ud_{\mathrm{TV}}(\mathcal{G}(0, \Sigma),\mathcal{G}(0,\Sigma(s^{\epsilon}+t^{\epsilon})))=0.
\]
From \eqref{C2} we obtain $\|\varphi(s^\epsilon+t^\epsilon,x)\|\leq \|x\|e^{-\delta (s^\epsilon+t^\epsilon)}$. Straightforward computations led us to deduce that
\[
\lim\limits_{\epsilon \to 0}\frac{\|\varphi(s^\epsilon+t^\epsilon,x)\|}{\sqrt{\epsilon}}=0.
\]
The latter together with item iii) of Lemma \ref{a1} imply
\[
\lim\limits_{\epsilon\rightarrow 0}
\ud_{\mathrm{TV}}\left(\mathcal{G}\left(\frac{\varphi(s^{\epsilon}+t^{\epsilon},x)}{\sqrt{\epsilon}}, \Sigma\right),\mathcal{G}(0,\Sigma)\right)=0.
\]
Therefore, from inequality \eqref{rty} we obtain 
\[
\lim\limits_{\epsilon\rightarrow 0}
\ud_{\mathrm{TV}}(\mathcal{G}(0,\epsilon \Sigma),z^{\epsilon}(s^\epsilon+t^\epsilon,x))=0
\]
Now, using the same ideas as in Proposition \ref{a18} (much easier since
$t^\epsilon\gg t_{\mathrm{mix}}^\epsilon$) we deduce 
\[
\lim\limits_{\epsilon\rightarrow 0}\ud_{\mathrm{TV}}(z^{\epsilon}(s^\epsilon+t^\epsilon,x),x^{\epsilon}(s^\epsilon+t^\epsilon,x))=0.
\]
From inequality \eqref{plo}, it remains to prove that
\[
\lim\limits_{\epsilon\rightarrow 0}\ud_{\mathrm{TV}}(x^{\epsilon}(s^\epsilon+t^\epsilon,x),\mu^{\epsilon})=0.
\]
Notice that
\[
\ud_{\mathrm{TV}}(x^{\epsilon}(s+t,x),\mu^{\epsilon})\leq \int\limits_{\mathbb{R}^d}{\ud_{\mathrm{TV}}(x^{\epsilon}(s+t,x),x^{\epsilon}(s+t,\bar{x}))\mu^{\epsilon}(\ud \bar{x})}.
\]
Since the stochastic differential equation associated to $\{y^\epsilon(t):t\geq 0\}$ {\it{is not homogeneous}} we should improve the notation as we did in the beginning of
Subsection \ref{gera}. Following such notation,  we always use $T=t^\epsilon$. Therefore, for simplicity, we can omit as we did in Proposition \ref{a18}.
Then
\begin{align*}
\int\limits_{\mathbb{R}^d}& \ud_{\mathrm{TV}}(x^{\epsilon}(s+t,x),x^{\epsilon}(s+t,\bar{x}))\mu^{\epsilon}(\ud \bar{x})\leq \\
& \ud_{\mathrm{TV}}(
x^{\epsilon}(s,x^{\epsilon}(t,x)),z^{\epsilon}(s,x^{\epsilon}(t,x)))+
 \ud_{\mathrm{TV}}(
z^{\epsilon}(s,x^{\epsilon}(t,x)),z^{\epsilon}(s,z^{\epsilon}(t,x)))+\\
 & \int\limits_{\mathbb{R}^d}{
\ud_{\mathrm{TV}}(
z^{\epsilon}(s,z^{\epsilon}(t,x)),z^{\epsilon}(s,z^{\epsilon}(t,\bar{x})))
 \mu^{\epsilon}(\ud \bar{x})}+\\
 & \int\limits_{\mathbb{R}^d}{
\ud_{\mathrm{TV}}(
z^{\epsilon}(s,z^{\epsilon}(t,\bar{x})),x^{\epsilon}(s,x^{\epsilon}(t,\bar{x})))
 \mu^{\epsilon}(\ud \bar{x})}.
\end{align*}
Again, using the same ideas as in Proposition \ref{a18} and Proposition \ref{a19} (much easier since
$t^\epsilon\gg t_{\mathrm{mix}}^\epsilon$) we deduce 
\[
\lim\limits_{\epsilon\rightarrow 0}
\ud_{\mathrm{TV}}(
x^{\epsilon}(s^\epsilon,x^{\epsilon}(t^\epsilon,x)),z^{\epsilon}(s^\epsilon,x^{\epsilon}(t^\epsilon,x)))=0.
\]
and
\[
\lim\limits_{\epsilon\rightarrow 0}
\ud_{\mathrm{TV}}(
z^{\epsilon}(s^\epsilon,x^{\epsilon}(t^\epsilon,x)),z^{\epsilon}(s^\epsilon,z^{\epsilon}(t^\epsilon,x)))=0.
\]
Fix $R>0$. We split the remainders integrals as follows
\begin{align*}
\int\limits_{\mathbb{R}^d}&
\ud_{\mathrm{TV}}(
z^{\epsilon}(s,z^{\epsilon}(t,x)),z^{\epsilon}(s,z^{\epsilon}(t,\bar{x})))
 \mu^{\epsilon}(\ud \bar{x})
\leq \\
& \int\limits_{\|\bar{x}\|\leq R}{
\ud_{\mathrm{TV}}(
z^{\epsilon}(s,z^{\epsilon}(t,x)),z^{\epsilon}(s,z^{\epsilon}(t,\bar{x})))
 \mu^{\epsilon}(\ud \bar{x})}+ \mu^\epsilon(\|x\|>R)
\end{align*}
and
\begin{align*}
\int\limits_{\mathbb{R}^d}&
\ud_{\mathrm{TV}}(
z^{\epsilon}(s,z^{\epsilon}(t,\bar{x})),x^{\epsilon}(s,x^{\epsilon}(t,\bar{x})))
 \mu^{\epsilon}(\ud \bar{x})
\leq \\
& \int\limits_{\|\bar{x}\|\leq R}
\ud_{\mathrm{TV}}(
z^{\epsilon}(s,z^{\epsilon}(t,\bar{x})),x^{\epsilon}(s,x^{\epsilon}(t,\bar{x})))
 \mu^{\epsilon}(\ud \bar{x})+ \mu^\epsilon(\|x\|>R).
\end{align*}
Notice that
\[
 \int\limits_{\|\bar{x}\|\leq R}{
\ud_{\mathrm{TV}}(
z^{\epsilon}(s,z^{\epsilon}(t,x)),z^{\epsilon}(s,z^{\epsilon}(t,\bar{x})))
 \mu^{\epsilon}(\ud \bar{x})}\leq
 \kappa(R,\|x\|) \frac{1}{\sqrt{\epsilon s}}e^{-\delta(t+s)},
\]
where $\kappa(R,\|x\|)$ is a non-negative constant
and $\delta>0$ comes from  \eqref{C2}.
By taking $t^\epsilon=\frac{1}{\epsilon^{\nicefrac{1}{8}}}\gg t_{\mathrm{mix}}^\epsilon$ and $s^\epsilon=\frac{1}{\epsilon^{\nicefrac{1}{2019}}}$ we obtain
$e^{-\delta {t^\epsilon}}
=o(\sqrt{\epsilon s^{\epsilon}})$. Therefore,
\[
\lim\limits_{\epsilon\rightarrow 0}
\int\limits_{\|\bar{x}\|\leq R}{
\ud_{\mathrm{TV}}(
z^{\epsilon}(s^\epsilon,z^{\epsilon}(t^\epsilon,x)),z^{\epsilon}(s^\epsilon,z^{\epsilon}(t^\epsilon,\bar{x})))
 \mu^{\epsilon}(\ud \bar{x})}=0.
\]
Following the same ideas as in the proof of Proposition \ref{a18}, we deduce 
\[
\lim\limits_{\epsilon \to 0}\int\limits_{\|\bar{x}\|\leq R}
\ud_{\mathrm{TV}}(
z^{\epsilon}(s^\epsilon,z^{\epsilon}(t^\epsilon,\bar{x})),x^{\epsilon}(s^\epsilon,x^{\epsilon}(t^\epsilon,\bar{x})))
 \mu^{\epsilon}(\ud \bar{x})=0.
\]
Now, we only need to prove that $\mu^\epsilon(\|x\|>R)$ is negligible when $\epsilon \to 0$. Following the same ideas in \cite{PSXZ} (page $122$, Section $5$, Step $1$), the invariant measure $\mu^\epsilon$ has finite $p$-moments for any $p\geq 0$. Moreover, we have
$\int\limits_{\mathbb{R}^d}{\|x\|^2}\mu^\epsilon(\ud x)\leq \frac{\epsilon d}{\delta}$.
Indeed,  from inequality  \eqref{zeroth} we have
\begin{equation*}
\mathbb{E}\left[\|x^{\epsilon}(t,x)\|^2\right] \leq  \|x\|^2e^{-2\delta t}+\frac{d\epsilon^2}{2\delta}\quad \textrm{ for any } t\geq 0 \textrm{ and } x\in \mathbb{R}^d. 
\end{equation*}
For any two numbers $a,b\in \mathbb{R}$, denote by 
$a\wedge b$ the minimum between $a$ and $b$. Recall that 
\begin{itemize}
\item[i)] If $a\leq b$ then $a \wedge c \leq b \wedge c $ for any $c\in \mathbb{R}$.
\item[ii)] $(a+b) \wedge c \leq a \wedge c+b \wedge c$ for any $a,b,c\geq 0$.
\end{itemize}
Notice that for any $t\geq 0$, $n\in \mathbb{N}$ and $x\in \mathbb{R}^d$ we have
\begin{align*}
\mathbb{E}\left[\|x^{\epsilon}(t,x)\|^2 \wedge n\right]\leq \mathbb{E}\left[\|x^{\epsilon}(t,x)\|^2\right ]\wedge n.
\end{align*}
Then
\begin{align*}
\mathbb{E}\left[\|x^{\epsilon}(t,x)\|^2 \wedge n\right]\leq \left(\|x\|^2e^{-2\delta t}+\frac{d\epsilon^2}{2\delta}\right)\wedge n \leq (\|x\|^2e^{-2\delta t})\wedge n+\frac{d\epsilon^2}{2\delta}\wedge n
\end{align*}
for any $t\geq 0$, $n\in \mathbb{N}$  and  $x\in \mathbb{R}^d$.
Integrating this inequality against 
$\mu^\epsilon(\ud x)$ we obtain
\begin{align*}
\int\limits_{\mathbb{R}^d}\left(\|x\|^2 \wedge n\right) \mu^{\epsilon}(\ud x)\leq 
\int\limits_{\mathbb{R}^d} \left( (\|x\|^2e^{-2\delta t})\wedge n \right)\mu^{\epsilon}(\ud x)+\frac{d\epsilon^2}{2\delta}\wedge n
\end{align*} 
for any $t\geq 0$ and  $n\in \mathbb{N}$. Passing to the limit first as $t\rightarrow \infty$ and using the Dominated Convergence Theorem we have
\begin{align*}
\int\limits_{\mathbb{R}^d}\left(\|x\|^2 \wedge n\right) \mu^{\epsilon}(\ud x)\leq \frac{d\epsilon^2}{2\delta}\wedge n
\end{align*}
for any $n\in \mathbb{N}$. Now, taking $n\rightarrow \infty$ and using the Monotone Convergence Theorem we have
\begin{align*}
\int\limits_{\mathbb{R}^d}\|x\|^2  \mu^{\epsilon}(\ud x)\leq \frac{d\epsilon^2}{2\delta}.
\end{align*}
The latter  together with the Chebyshev inequality imply 
\[\mu^{\epsilon}\left(\|x\|\geq R\right)\leq \frac{d\epsilon^2}{2R^2\delta}\quad  \textrm{ for any } R>0.\]
\end{proof}

\subsubsection{Proof of Theorem \ref{main}}
Now, we are ready to prove Theorem \ref{main}. To stress the fact that Theorem \ref{main} is just a
consequence of what we have proved up to here, we state this as a Lemma.

\begin{lemma}
Assume that 
 \eqref{C2} and   \eqref{C3} hold.
Let $\{x^\epsilon(t,x_0): t \geq 0\}$ be the solution of \eqref{SEDO}
and denote by $\mu^\epsilon$ the unique invariant probability measure for the evolution given by \eqref{SEDO}.
Denote by 
\[
d^{\epsilon}(t)=\ud_{\mathrm{TV}}(x^{\epsilon}(t,x_0),\mu^\epsilon) \quad
\textrm{for any } t\geq 0
\] 
the total variation distance between the law of the random variable $x^{\epsilon}(t,x_0)$ and its invariant probability $\mu^\epsilon$.
Consider the cut-off time $t_{\mathrm{mix}}^\epsilon$  given by \eqref{mix} and the time window given by
\eqref{win}. Let $x_0\not=0$. Then for
any $c\in \mathbb{R}$ we have
\[
\lim\limits_{\epsilon\rightarrow 0}{\Big|d^{\epsilon}(t_{\mathrm{mix}}^\epsilon+cw^\epsilon)-D^{\epsilon}(t_{\mathrm{mix}}^\epsilon+cw^\epsilon)\Big|}=0,
\]
where
\begin{equation*}
D^{\epsilon}(t)=\ud_{\mathrm{TV}}\left(\mathcal{G}\left(\frac{(t-\tau)^{\ell-1}}{e^{\lambda (t-\tau)}\sqrt{\epsilon}}\Sigma^{-\nicefrac{1}{2}}\sum\limits_{k=1}^{m}e^{i\theta_k (t-\tau)}v_k,\mathrm{I}_d \right),\mathcal{G}(0,\mathrm{I}_d) \right)
\end{equation*}
for any $t\geq \tau$
with $m$, $\lambda$, $\ell$, $\tau$, $\theta_1,\ldots, \theta_m $, $v_1,\ldots,v_m$ are the constants and vectors  associated to $x_0$
in {\em{Lemma \ref{asymp}}},
and the matrix $\Sigma$ is the unique solution of the matrix Lyapunov equation:
\begin{equation*}
DF(0)X+X (DF(0))^*=\mathrm{I}_d.
\end{equation*}
\end{lemma}

\begin{proof}
Firstly, from Lemma \ref{ec4} we have that there exists a unique invariant probability measure for the  evolution \eqref{SEDO}. Let call the invariant measure by $\mu^\epsilon$.
From Lemma \ref{goodine} together with Proposition \ref{a18}, Proposition \ref{a19} and Proposition \ref{a20} we deduce
\[
\Big|d^{\epsilon}(t_{\mathrm{mix}}^\epsilon+cw^\epsilon)-\bar{D}^{\epsilon}(t_{\mathrm{mix}}^\epsilon+cw^\epsilon)\Big|=o(1) \quad \textrm{ as } \epsilon \to 0.
\]
From the triangle inequality we obtain
\[
\begin{split}
\Big|d^{\epsilon}(t_{\mathrm{mix}}^\epsilon+&cw^\epsilon)-D^{\epsilon}(t_{\mathrm{mix}}^\epsilon+cw^\epsilon)\Big| \leq \\
&
\Big|{D}^{\epsilon}(t_{\mathrm{mix}}^\epsilon+cw^\epsilon)-\bar{D}^{\epsilon}(t_{\mathrm{mix}}^\epsilon+cw^\epsilon)\Big|+o(1) \quad \textrm{ as } \epsilon \to 0.
\end{split}
\]
The latter together with Proposition \ref{prop2} allows to deduce the statement.
\end{proof}

\appendix
\section{Properties of the Total Variation Distance for Gaussian Distributions}\label{pgdis}
Recall that  $\mathcal{G}{\left(v,\Xi \right)}$ denotes the Gaussian distribution in $\mathbb{R}^d$ with vector mean $v$ and positive definite covariance matrix $\Xi$.
Since the proofs are straightforward, we left most of details to the interested reader.

\begin{lemma}\label{a1}
Let $v,\tilde{v} \in \mathbb{R}^d$ be two fixed vectors and  $\Xi,\tilde{\Xi}$ be two fixed   symmetric positive definite $d\times d$ matrices. Then
\begin{itemize}
\item[i)] For any scalar $c\not=0$ we have
\[
\ud_{\mathrm{TV}}\left(\mathcal{G}{\left(cv,{c}^2\Xi\right)},\mathcal{G}{\left(c\tilde{v},{c}^2\tilde{\Xi}\right)}\right)=
\ud_{\mathrm{TV}}\left(\mathcal{G}{\left(v,\Xi\right)},\mathcal{G}{\left(\tilde{v},\tilde{\Xi}\right)}\right).
\]
\item[ii)]
$
\ud_{\mathrm{TV}}\left(\mathcal{G}{\left(v,\Xi \right)},\mathcal{G}{\left(\tilde{v},\tilde{\Xi} \right)}\right)=
\ud_{\mathrm{TV}}\left(\mathcal{G}{\left(v-\tilde{v},\Xi\right)},\mathcal{G}{\left(0,\tilde{\Xi} \right)}\right).
$
\item[iii)]
$
\ud_{\mathrm{TV}}\left(\mathcal{G}{\left(v,\Xi \right)},\mathcal{G}{\left(\tilde{v},{\Xi} \right)}\right)=
\ud_{\mathrm{TV}}\left(\mathcal{G}{\left(\Xi^{-\nicefrac{1}{2}}v,\mathrm{I}_d\right)},\mathcal{G}{\left(\Xi^{-\nicefrac{1}{2}}\tilde{v},\mathrm{I}_d \right)}\right).
$
\item[iv)]
$
\ud_{\mathrm{TV}}\left(\mathcal{G}{\left(0,{\Xi} \right)},\mathcal{G}{\left(0,\tilde{\Xi}\right)}\right)=
\ud_{\mathrm{TV}}\left(\mathcal{G}{\left(0,\tilde{\Xi}^{-\nicefrac{1}{2}}\Xi\tilde{\Xi}^{-\nicefrac{1}{2}} \right)},\mathcal{G}{\left(0,\mathrm{I}_d\right)}\right).
$
\end{itemize}
\end{lemma}
\begin{proof} 
The proofs follow  from the characterisation of the total variation distance between two probability measures with
densities, i.e.,
\[
\ud_{\mathrm{TV}}(\mathbb{P}_1,\mathbb{P}_2)=\frac{1}{2}\int\limits_{\mathbb{R}^d}\left|f_{1}(x)-f_{2}(x) \right|\ud x,
\]
where $f_1$ and $f_2$ are the densities of $\mathbb{P}_1$ and $\mathbb{P}_2$, respectively, and using the Change of Variable Theorem.
\end{proof}

\begin{lemma}\label{a2}
For any $v\in \mathbb{R}^d$ we have
\[
\ud_{\mathrm{TV}}\left(\mathcal{G}{\left(v,\mathrm{I}_d \right)},\mathcal{G}{\left(0,\mathrm{I}_d\right)}\right)=
\sqrt{\frac{2}{{\pi}}}\int\limits_{0}^{\nicefrac{\|v\|}{2}}e^{-\frac{x^2}{2}}\ud x\leq  \frac{1}{\sqrt{2\pi}}\|v\|.
\]
\end{lemma}
\begin{proof} 
The proof in dimension one is a straightforward computation. 
We left the details to the interested reader.
For dimension bigger than one, the idea is to reduce the proof to dimension one. To do that, we use the following fact:
for any $v,\tilde{v}\in \mathbb{R}^d$ such that $\|v\|=\|\tilde{v}\|$ there exists an orthogonal matrix $A$ such that
$\tilde{v}=A(v)$.
Recall  that the law of $\mathcal{G}{\left(0,\mathrm{I}_d\right)}$ is invariant under orthogonal transformations, i.e., 
$O\mathcal{G}{\left(0,\mathrm{I}_d\right)}=\mathcal{G}{\left(0,\mathrm{I}_d\right)}$ for any orthogonal matrix $O$.
Then for any $v,\tilde{v}\in \mathbb{R}^d$ we have
\[
\begin{split}
\ud_{\mathrm{TV}}\left(\mathcal{G}{\left(\tilde{v},\mathrm{I}_d \right)},\mathcal{G}{\left(0,\mathrm{I}_d\right)}\right)
&=
\ud_{\mathrm{TV}}\left(\mathcal{G}{\left(A{v},\mathrm{I}_d \right)},\mathcal{G}{\left(0,\mathrm{I}_d\right)}\right)
=
\ud_{\mathrm{TV}}\left(\mathcal{G}{\left(A{v},\mathrm{I}_d \right)},A\mathcal{G}{\left(0,\mathrm{I}_d\right)}\right)\\
&=\ud_{\mathrm{TV}}\left(A\mathcal{G}{\left({v},\mathrm{I}_d \right)},A\mathcal{G}{\left(0,\mathrm{I}_d\right)}\right)
=\ud_{\mathrm{TV}}\left(\mathcal{G}{\left({v},\mathrm{I}_d \right)},\mathcal{G}{\left(0,\mathrm{I}_d\right)}\right),
\end{split}
\]
where the last equality follows 
from the characterisation of the total variation distance between two probability measures with
densities and the Change of Variable Theorem.
The latter allows us to reduce the proof to dimension one by observing that the vectors $v$ and $(\|v\|,0,\ldots,0)^{*}\in \mathbb{R}^d$ have the same norm and  the statement follows from a straightforward computation. 
\end{proof}

\begin{lemma}\label{a4}
Let $\{v_{\epsilon}:\epsilon>0\}\subset \mathbb{R}^d$ such that $\lim\limits_{\epsilon\rightarrow 0}{v_{\epsilon}}=v\in \mathbb{R}^d$. Then
\[
\lim\limits_{\epsilon\rightarrow 0}
\ud_{\mathrm{TV}}\left(\mathcal{G}{\left(v_{\epsilon},\mathrm{I}_d \right)},\mathcal{G}{\left(0,\mathrm{I}_d\right)}\right)=
\ud_{\mathrm{TV}}\left(\mathcal{G}{\left(v,\mathrm{I}_d\right)},\mathcal{G}{\left(0,\mathrm{I}_d\right)}\right).
\]
\end{lemma}
\begin{proof}
The idea of the proof follows from Lemma \ref{a2} together with the Dominated Convergence Theorem.
\end{proof}

\begin{lemma}\label{a5}
Let $\{v_{\epsilon}:{\epsilon>0}\}\subset \mathbb{R}^d$ such that $\lim\limits_{\epsilon\rightarrow 0}{\|v_{\epsilon}\|}=+\infty$. Then
\[
\lim\limits_{\epsilon\rightarrow 0}\ud_{\mathrm{TV}}\left(\mathcal{G}{\left(v_{\epsilon},\mathrm{I}_d \right)},\mathcal{G}{\left(0,\mathrm{I}_d\right)}\right)=1.
\]
\end{lemma}
\begin{proof}
The idea of the proof follows from Lemma \ref{a2} together with the Dominated Convergence Theorem.
\end{proof}

\begin{lemma}\label{a6}
Let $\mathcal{S}_d$ denotes the set of $d\times d$ symmetric and positive definite matrices.
Let $\{\Xi_{\epsilon}:{\epsilon>0}\}\subset \mathcal{S}_d$ such that $\lim\limits_{\epsilon\rightarrow 0}{\Xi_{\epsilon}}=\Xi\in \mathcal{S}_d$. Then
\[
\lim\limits_{\epsilon\rightarrow 0}\ud_{\mathrm{TV}}\left(\mathcal{G}{\left(0,\Xi_{\epsilon} \right)},\mathcal{G}{\left(0,\Xi\right)}\right)=0.
\]
\end{lemma}
\begin{proof}
The proof follows from 
the characterisation of the total variation distance between two probability measures with
densities together with the Scheff\'e Lemma.
\end{proof}

For $m\in \mathbb{R}$, $\mathcal{N}{\left(m,1 \right)}$ denotes the Gaussian distribution on $\mathbb{R}$ with mean $m$ and unit variance .
\begin{lemma}\label{lusc} 
Let $\{v_t:t\geq 0\}\subset \mathbb{R}^d$. 
\begin{itemize}
\item[i)]
If $\limsup\limits_{t\rightarrow +\infty}\|v_t\|\leq C_0\in [0,+\infty)$ then 
\[
\limsup\limits_{t\rightarrow +\infty}\ud_{\mathrm{TV}}(\mathcal{G}(v_t,\mathrm{I}_d),\mathcal{G}(0,\mathrm{I}_d))\leq \ud_{\mathrm{TV}}(\mathcal{N}(C_0,1),\mathcal{N}(0,1)).
\]
\item[ii)] If $\liminf\limits_{t\rightarrow +\infty}\|v_t\|\geq  C_1\in [0,+\infty)$ then 
\[
\liminf\limits_{t\rightarrow +\infty}\ud_{\mathrm{TV}}(\mathcal{G}(v_t,\mathrm{I}_d),\mathcal{G}(0,\mathrm{I}_d))\geq  \ud_{\mathrm{TV}}(\mathcal{N}(C_1,1),\mathcal{N}(0,1)).
\]
\end{itemize}
\end{lemma}
\begin{proof}
From Lemma \ref{a2} we deduce 
\[
\ud_{\mathrm{TV}}(\mathcal{G}(v_t,\mathrm{I}_d),\mathcal{G}(0,\mathrm{I}_d))=\ud_{\mathrm{TV}}(\mathcal{N}(\|v_t\|,1),\mathcal{N}(0,1))=
\sqrt{\frac{2}{{\pi}}}\int\limits_{0}^{\nicefrac{\|v_t\|}{2}}e^{-\frac{x^2}{2}}\ud x
\]
which allows to reduce the proof for $d=1$. The proof proceeds from the following straightforward argument: after passing a subsequence, we use the continuity of the total variation distance (Lemma \ref{a4} and Lemma \ref{a6}) and the monotonicity property:
\[\ud_{\mathrm{TV}} (\N(m_1,1),\N(0,1))\leq
\ud_{\mathrm{TV}} (\N(m_2,1),\N(0,1))
\]
for any $0\leq |m_1|\leq |m_2|<+\infty$ in order to deduce item i) and item ii) of the statement.
\end{proof}

\section{The deterministic dynamical system}
\label{B.tomate}
In this section we present a proof of Lemma \ref{asymp}. We start analysing the linear differential equation associated to the linearisation of the non-linear deterministic differential equation (\ref{EDO}) around the hyperbolic fixed point $0$.

\begin{lemma}
\label{asy}
Assume that 
 \eqref{C1}  holds.
Then for any $x_0\in \mathbb{R}^{d}\setminus\{0\}$  there exist $\lambda:=\lambda(x_0)>0$, $\ell:=\ell(x_0),\; m:=m(x_0) \in \bb \{1,\ldots,d\}$, $\theta_1:=\theta_1(x_0),\dots,\theta_m:=\theta_m(x_0) \in [0,2\pi)$ and $v_1:=v_1(x_0),\dots,v_m:=v_m(x_0)$ in $\bb C^d$ linearly independent such that
\[
\lim_{t \to +\infty} \Big\| \frac{e^{\lambda t}}{t^{\ell-1}} e^{-DF(0)t}x_0 - \sum_{k=1}^m e^{i\theta_k t} v_k \Big\|=0.
\]
\end{lemma}

\begin{proof}
Write $\Lambda=DF(0)$  and let $t\geq 0$. 
We will use the Putzer spectral method to ``compute" $e^{-\Lambda t}x_0$.
By \eqref{C1}, all eigenvalues of $\Lambda$ have positive real part. Denote by $\{\phi(t,x): t \geq 0\}$ the solution of the linear system:
\[
\left\{
\begin{array}{r@{\;=\;}l}
\tfrac{\ud}{\ud t} \phi(t) & -\Lambda \phi(t)\quad \textrm{ for } t\geq 0,\\
\phi(0) & x.
\end{array}
\right.
\]
Let $(w_{j,k}: j=1,\dots,N; k=1,\dots, N_j)$ be a Jordan basis of $-\Lambda$, that is,
\[
-\Lambda w_{j,k} = -\lambda_j w_{j,k} + w_{j,k+1}
\]
for any $j=1,\dots,N; k=1,\dots, N_j$.
In this formula we use the convention $w_{j,N_j+1} =0$. Since
$(w_{j,k}: j=1,\dots,N; k=1,\dots, N_j)$ is a basis of $\mathbb{R}^d$,
the decomposition
\[
\phi(t,x) = \sum_{j=1}^{N}\sum_{k=1}^{N_j} \phi_{j,k}(t,x) w_{j,k}
\]
defines the functions $\phi_{j,k}(t,x)$ in a unique way. Then
\[
\sum_{j=1}^{N}\sum_{k=1}^{N_j} \tfrac{\ud}{\ud t}\phi_{j,k}(t,x) w_{j,k} = \sum_{j=1}^{N}\sum_{k=1}^{N_j} \phi_{j,k}(t,x) \big(-\lambda_j w_{j,k} +w_{j,k+1}\big),
\]
and the aforementioned uniqueness implies 
\begin{equation*}
\label{comino}
\tfrac{\ud}{\ud t} \phi_{j,k}(t,x) = -\lambda_j \phi_{j,k}(x,t) + \phi_{j,k-1}(t,x)
\end{equation*}
for any $j=1,\dots,N; k=1,\dots, N_j$,
where we use the convention $\phi_{j,0}(t,x) =0$. In addition, we have that $\phi_{j,k}(0,x) = x_{j,k}$, where
\[
x = \sum_{j=1}^{N}\sum_{k=1}^{N_j}  x_{j,k} w_{j,k}.
\]
For each $j$, the system of equations for $\{\phi_{j,k}(t,x): k=1,\dots,N_j\}$ is autonomous, as well as the equation for $\phi_{j,1}(t,x)$. Notice that
\[
\phi_{j,1}(t,x) = x_{j,1}e^{-\lambda_j t}
\]
and by the method of variation of parameters, for $k=2,\dots,N_j$ we have 
\[
\phi_{j,k}(t,x) = x_{j,k}e^{-\lambda_j t} + \int_0^t e^{-\lambda_j (t-s)} \phi_{j,k-1}(s,x) \ud s.
\]
Applying this formula for $k=2$ we see 
\[
\phi_{j,2}(t,x) = x_{j,2} e^{-\lambda_j t} + x_{j,1} t e^{-\lambda_j t}
\]
and from this expression we can guess and check the formula
\[
\phi_{j,k}(t,x) = \sum_{i=1}^k x_{j,i}  \frac{t^{k-i} e^{-\lambda_j t}}{(k-i)!}.
\]
We conclude that
\begin{equation}
\label{calabaza}
\phi(t,x)  = \sum_{j=1}^{\vphantom{N_j}N} \sum_{k=1}^{N_j} \sum_{i=1}^{\vphantom{N_j} k}  \frac{t^{k-i} e^{-\lambda_j t}}{(k-i)!} x_{j,i} w_{j,k}.
\end{equation}
With this expression in hand, we are ready to prove Lemma \ref{asy}. Let $x_0 \in \bb R^d$ be fixed.
Assume that $x_0 \neq 0$ and write
\[
x_0 =\sum_{j=1}^{N}\sum_{k=1}^{N_j}  x_{j,k}^0 w_{j,k}.
\]
Take
\[
\lambda = \min\{ \mathrm{Re}(\lambda_j): x^0_{j,k} \neq 0 \text{ for some } k\}
\]
and define
\[
J_0 = \{j : \mathrm{Re}(\lambda_j) = \lambda \text{ and } x_{j,k}^0 \neq 0 \text{ for some } k\}.
\]
In other words, we identify in \eqref{calabaza} the smallest exponential rate of decay and we collect in $J_0$ all the indices with that exponential decay.
Now,  define
\[
\ell = \max\{ N_j-k: j \in J_0 \text{ and } x_{j,k}^0 \neq 0\}
\]
and
\[
J = \{j \in J_0: x_{j,N_j-\ell}^0 \neq 0\}.
\]
We see that for $j \in J$,
\[
\lim_{t \to \infty} \big|\phi_{j,N_j}(t,x_0)\big|\frac{e^{\lambda t}}{t^\ell} = \frac{\big|x_{j,N_j-\ell}\big|}{\ell !},
\]
while for $j \notin J$ and $k$ arbitrary or $j \in J$ and $k \neq N_j$,
\[
\lim_{t \to \infty} \big|\phi_{j,k}(t,x_0)\big|\frac{e^{\lambda t}}{t^\ell} = 0.
\]
Therefore,
\[
\lim_{t \to \infty} \Big\| \frac{e^{\lambda t}}{t^\ell} \phi(t,x_0) - \sum_{j \in J} \frac{e^{-(\lambda_j-\lambda)t} }{\ell !} x_{j,N_j-\ell} w_{j,N_j}\Big\| = 0.
\]
Let $m =\# J$ and let $\sigma: \{1,\dots,m\} \to J$ be a numbering of $J$.
By definition of $\lambda$ and $J$, the numbers $\lambda_j-\lambda$ are imaginary. Therefore, Lemma \ref{asy} is proved choosing $\theta_k = i(\lambda_{\sigma_k}-\lambda)$ and $v_k = \frac{x_{\sigma_k,N_{\sigma_k}-\ell} w_{\sigma_k,N_{\sigma_k}}}{\ell !}$.
\end{proof}

Now, we are ready to prove Lemma \ref{asymp}. The proof is based in the Hartman-Grobman Theorem
(see Theorem(Hartman) page $127$ of \cite{PE} or the celebrated paper of P.~Hartman \cite{HA})
that guarantees that the conjugation around the hyperbolic fixed point $0$ of (\ref{EDO}) is $\mathcal{C}^1$-local diffeomorphism under some resonance conditions which are fulfilled when all the eigenvalues of the matrix $DF(0)$ have negative (or positive) real part.

\begin{lemma}
\label{asympgeral}
Assume that \eqref{C1} holds. Then
for any $x_0\in \mathbb{R}^{d}\setminus\{0\}$  there exist $\lambda:=\lambda(x_0)>0$, $\ell:=\ell(x_0), m:=m(x_0) \in \bb N$, $\theta_1:=\theta_1(x_0),\dots,\theta_m:=\theta_m(x_0) \in [0,2\pi)$, $v_1:=v_1(x_0),\dots,v_m:=v_m(x_0)$ in $\bb C^d$ linearly independent and $\tau:=\tau(x_0)>0$ such that
\[
\lim_{t \to +\infty} \Big\| \frac{e^{\lambda t}}{t^{\ell-1}} \varphi(t+\tau,x_0) - \sum_{k=1}^m e^{i\theta_k t} v_k \Big\|=0.
\]
\end{lemma}

\begin{proof}
Since all the eigenvalues of $DF(0)$ have real positive real part, there exist open sets $U,V$ around the hyperbolic fixed  point zero and $h:U\rightarrow V$ a $C^1{(U,V)}$ homeomorphism such that $h(0)=0$ and $h^{}(x)=x+o(\|x\|)$ as $\|x\|\rightarrow 0$ such that
$\varphi(t,x)=h^{-1}(e^{-DF(0)t}h(x))$  for any  $t\geq 0$ and $x\in U$.
From \eqref{C1} we obtain 
\[
\|\varphi(t,x)\| \leq \|x\| e^{-\delta t}\quad \text{ for any } x \in \bb R^d \text{ and any } t \geq 0.
\]
Observe that there exists $\tau:=\tau(x_0)>0$ such that $\varphi(t,x_0)\in U$ for any $t\geq \tau$. Then
\[
\varphi(t+\tau,x_0)=\varphi(t,x_\tau)=h^{-1}(e^{-DF(0)t}h(x_\tau))\quad \textrm{ for any }t\geq 0.\]

Let $\tilde{x}:=h(x_\tau)$. By Lemma \ref{asy}
there exist $\lambda(\tilde{x}):=\lambda >0$, $\ell(\tilde{x}):=\ell, m(\tilde{x}):=m \in \bb N$, $\theta_1(\tilde{x}):=\theta_1,\dots,\theta_m(\tilde{x}):=\theta_m \in [0,2\pi)$ and $v_1(\tilde{x}):=v_1,\dots,v_m(\tilde{x}):=v_m$ in $\bb C^d$ linearly independent such that
\[
\lim_{t \to +\infty} \Big\| \frac{e^{\lambda t}}{t^{\ell-1}} e^{-DF(0)t}\tilde{x} - \sum_{k=1}^m e^{i\theta_k t} v_k \Big\|=0.
\]
From the triangle inequality we obtain
\begin{equation}\label{lola}
\begin{split}
\Big\| \frac{e^{\lambda t}}{t^{\ell-1}} \varphi(t+\tau,x_0) - \sum_{k=1}^m e^{i\theta_k t} v_k \Big\|  \leq &
\Big\|\frac{e^{\lambda t}}{t^{\ell-1}} \varphi(t+\tau,x_0)- \frac{e^{\lambda t}}{t^{\ell-1}} e^{-DF(0)t}\tilde{x}\Big\|\\
&  +\Big\| \frac{e^{\lambda t}}{t^{\ell-1}} e^{-DF(0)t}\tilde{x} - \sum_{k=1}^m e^{i\theta_k t} v_k \Big\|.\end{split}
\end{equation}
Observe that
\[
\begin{split}
&\frac{e^{\lambda t}}{t^{\ell-1}}\Big\| \varphi(t+\tau,x_0)- e^{-DF(0)t}\tilde{x}\Big\|
=
\frac{e^{\lambda t}}{t^{\ell-1}}\Big\| h^{-1}(e^{-DF(0)t}h(x_\tau))- e^{-DF(0)t}\tilde{x}\Big\|\\
&\hspace{3cm}=
\frac{e^{\lambda t}}{t^{\ell-1}} o\left(\Big\|e^{-DF(0)t}\tilde{x}\Big\|\right)=\frac{e^{\lambda t}\Big\|e^{-DF(0)t}\tilde{x}\Big\|}{t^{\ell-1}} o(1)\\
&\hspace{3cm}\leq
\Big\| \frac{e^{\lambda t}}{t^{\ell-1}} e^{-DF(0)t}\tilde{x} - \sum_{k=1}^m e^{i\theta_k t} v_k \Big\| o(1)+
\left( \sum_{k=1}^m \|v_k\|\right)o(1),
\end{split}
\]
where $o(1)$ goes to zero as $t$ goes by.
The latter together with inequality \eqref{lola} and Lemma \ref{asy} allows us to deduce 
\begin{align*}
\lim\limits_{t\rightarrow +\infty}\Big\| \frac{e^{\lambda t}}{t^{\ell-1}} \varphi(t+\tau,x_0) - \sum_{k=1}^m e^{i\theta_k t} v_k \Big\|=0.
\end{align*}
\end{proof}

\begin{lemma}\label{pepino} 
Assume that \eqref{C1} holds. Let $\delta_\epsilon=o(1)$. Then
\[
\lim\limits_{\epsilon\rightarrow 0}{\frac{\delta_\epsilon\|\varphi(t^\epsilon_{\mathrm{mix}}+\delta_\epsilon+cw^{\epsilon},x_0)\|^2}{{\epsilon}}}=0 \quad \textrm{ for any } c\in \mathbb{R}.
\]
\end{lemma}

\begin{proof}
Remember that
\[
t_{\mathrm{mix}}^\epsilon = \frac{1}{2\lambda} \ln \left(1/\epsilon\right) + \frac{\ell-1}{\lambda} \ln\left(
\ln \left(1/\epsilon\right)\right)+\tau,
\]
and 
\[
w^\epsilon=\frac{1}{\lambda}+o(1),
\]
where $\lambda$, $\ell$ and $\tau$ are the constants associated to $x_0$ in Lemma \ref{asymp} and 
$o(1)$ goes to zero as $\epsilon \to 0$.
Define
$
t^\epsilon := t_{\mathrm{mix}}^\epsilon -\tau+\delta_\epsilon+cw^\epsilon.
$
Note
\begin{align*}
\frac{1}{\sqrt{\epsilon}}\Big\|\varphi(t^\epsilon+\tau,x_0)\Big\|\leq &
\frac{(t^\epsilon)^{\ell-1}}{e^{\lambda t^\epsilon}\sqrt{\epsilon}}\Big\|\frac{e^{\lambda t^\epsilon}}{(t^\epsilon)^{\ell-1}} \varphi(t^\epsilon+\tau,x_0)- \sum_{k=1}^m e^{i\theta_k t_\epsilon} v_k \Big\|+\\
&\frac{(t^\epsilon)^{\ell-1}}{e^{\lambda t^\epsilon}\sqrt{\epsilon}}\sum_{k=1}^m \|v_k\|.
\end{align*}
From the last inequality, using the fact that $\lim\limits_{\epsilon\rightarrow 0}\frac{(t^\epsilon)^{\ell-1}}{e^{\lambda t^\epsilon}\sqrt{\epsilon}}=\frac{e^{-c}}{(2\lambda)^{\ell-1}}$ and Lemma
\ref{asympgeral} we deduce the  desired result.
\end{proof}
\section{The stochastic dynamical system}\label{apend3}

In this Appendix we analyse the zeroth and first order approximations for the It\^o diffusion $\{x^{\epsilon}(t):t\geq 0\}$. Recall that 
$\{y(t):t\geq 0\}$ is the solution of the stochastic differential equation \eqref{dc},
and $\delta>0$ is the constant that appears in \eqref{C2}.

\begin{lemma}\label{uap}
Assume that \eqref{C2} holds.
For any $\eta>0$ and $t\in \left[0, \frac{\eta^2}{\epsilon d}\right)$ we have
\begin{align*}
\mathbb{P}\left(\sup\limits_{0\leq s\leq t}\|x^{\epsilon}(t)-\varphi(t)\|\geq \eta\right)\leq \frac{2d\epsilon^2 t}{\delta\left(\eta^2-\epsilon dt\right)^2} 
\end{align*}
and 
\begin{align*}
\mathbb{P}\left(\sup\limits_{0\leq s\leq t}\|\sqrt{\epsilon}y^{}(t)\|\geq \eta\right)\leq \frac{2d\epsilon^2 t}{\delta\left(\eta^2-\epsilon dt\right)^2}. 
\end{align*}
\end{lemma}

\begin{proof}
Let $\epsilon>0$ and $t\geq 0$ be fixed.
From \eqref{pin} we have
\begin{align}\label{pin5}
\ud \|x^\epsilon(t) -\varphi(t)\|^2 \leq & 
-2\delta \|x^{\epsilon}(t)-\varphi(t)\|^2\ud t+
2\sqrt \epsilon \<(x^\epsilon(t) -\varphi(t)), \ud B(t)\>\nonumber\\
&+ d \epsilon \ud t. 
\end{align}
Let $M^\epsilon(t):=2\sqrt{\epsilon}\left(x^{\epsilon}(t)-\varphi(t)\right)^{*}$ for every $t\geq 0$. Notice that 
\[
\left\{N^\epsilon(t):=\int_{0}^{t}M^{\epsilon} 
(s) \ud B(s): t\geq 0\right\}\quad \textrm{ is a local martingale. }
\]  
Then, there exists a sequence of increasing stopping times $\{\tau^{\epsilon}_n\}_{n\in \mathbb{N}}$ such that almost surely $\tau^{\epsilon}_n\uparrow \infty$ as $n$ goes to infinity and for each $n\in \mathbb{N}$, 
\[
\left\{N^{\epsilon,n}(t)=N^{\epsilon}\left({\min\{\tau^{\epsilon}_n,t\}} \right):t\geq 0\right\}
\quad \textrm{ is a true martingale.
}\] 
Taking expectation on \eqref{pin5} and using the fact that $\{N^{\epsilon,n}(t):t\geq 0\}$ is a zero-mean martingale, we deduce 
\begin{equation*}
\mathbb{E}\left[\|x^{\epsilon}\left(\min\{\tau^{\epsilon}_n,t\}\right)-\varphi\left({\min\{\tau^{\epsilon}_n,t\}}\right)\|^2\right]\leq 
\epsilon d{\min\{\tau^{\epsilon}_n,t\}}
\leq  \epsilon dt
\end{equation*}
for every $t\geq 0$. Consequently, by the well--known Fatou Lemma we obtain
\begin{equation*}
\mathbb{E}\left[\|x^{\epsilon}(t)-\varphi(t)\|^2\right]\leq 
\epsilon dt \quad \textrm{ for any } t\geq 0.
\end{equation*}
The latter implies
\[
\left\{N^\epsilon(t)=\int_{0}^{t}M^{\epsilon} 
(s) \ud B(s): t\geq 0\right\} \quad \textrm{ is a  true martingale}.
\]
From inequality \eqref{pin5} we have
\[
\|x^{\epsilon}(t)-\varphi(t)\|^2
\leq \epsilon dt+N^{\epsilon}(t)\quad \textrm{ for any } t\geq 0.
\]
For any $\eta>0$ and $0\leq t<\nicefrac{\eta^2}{(\epsilon d)}$ we have
\[
\mathbb{P}\left(\sup\limits_{0\leq s\leq t}\|x^{\epsilon}(s)-\varphi(s)\|^2\geq \eta^2\right)\leq\mathbb{P}\left(\sup\limits_{0\leq s\leq t}\|N^{\epsilon}(s)\|\geq \eta^2-\epsilon dt\right).
\]
From the Doob inequality for submartingales we obtain
\begin{align*}
\mathbb{P}\left(\sup\limits_{0\leq s\leq t}\|N^{\epsilon}(s)\|\geq \eta^2-\epsilon dt\right)\leq \frac{\mathbb{E}\left[\|N^{\epsilon}(t)\|^2\right]}{(\eta^2-\epsilon dt)^2}.
\end{align*}
The It\^o isometry allows us to deduce that
\[
\mathbb{E}\left[\|N^{\epsilon}(t)\|^2\right]=4\epsilon\
\int_{0}^{t}\mathbb{E}\left[\|x^{\epsilon}(s)-\varphi(s)\|^2\right]\ud s.
\]
From inequality \eqref{zeroth} we obtain
$
\mathbb{E}\left[\|N^{\epsilon}(t)\|^2\right]\leq \nicefrac{2d\epsilon^2 t}{\delta}
$.
Therefore
\begin{align*}
\mathbb{P}\left(\sup\limits_{0\leq s\leq t}\|x^{\epsilon}(s)-\varphi(s)\|\geq \eta\right)\leq \frac{2d\epsilon^2 t}{\delta(\eta^2-\epsilon dt)^2}
\end{align*}
for $0\leq t<\nicefrac{\eta^2}{(\epsilon d)}$. The proof for the second part proceeds from the same ideas as the first part. We left the details to the interested reader.
\end{proof}

\begin{proposition}\label{a16}
Assume that \eqref{C2} holds.
For $t\geq 0$, write $W(t):=\sup\limits_{0\leq s\leq t}{\|B(s)\|}$.
For any $t\geq 0$, the following holds true:
\begin{itemize}
\item[i)]
$\bb E \left[\|x^\epsilon(t) -\varphi(t)\|^2 \right]
\leq \frac{d \epsilon}{2\delta}$\;
and\;
$\bb E \left[\|y(t)\|^2 \right] 
\leq \frac{d }{2\delta}$.
\item[ii)]  For each $n\in \mathbb{N}$,
define $c_n:=\prod\limits_{j=0}^{n-1}{\left(d+2j\right)}$.
Then
\[
\mathbb{E}\left[ \left\|x^{\epsilon}(t)-\varphi(t)\right\|^{2n} \right]\leq \frac{c_n\epsilon^n}{2^n\delta^n}\quad \textrm{ and }\quad
\mathbb{E}\left[ \left\|y(t)\right\|^{2n} \right]\leq \frac{c_n}{2^n\delta^n}.
\]
\item[iii)] 
For any $0\leq r<\delta$ we have 
\[
\mathbb{E}\left[  \exp\left({r\frac{\left\|x^{\epsilon}(t)-\varphi(t)\right\|^2}{\epsilon}}\right) \right]<+\infty \quad
 \textrm{ and } \quad
 \mathbb{E}\left[  \exp\left({r{\left\|y(t)\right\|^2}}\right) \right]<+\infty.
\]
\item[iv)] 
Let $\delta_\epsilon\in (0,\nicefrac{\delta}{2}]$. Then 
\[
\mathbb{E}\left[  \exp\left({\delta_{\epsilon}\frac{\left\|x^{\epsilon}(t)-\varphi(t)\right\|^2}{\epsilon}} \right)\right]\leq
\exp\left(d{\delta_{\epsilon}t}\right)
\]
and
\[
\mathbb{E}\left[  \exp\left({\delta_{\epsilon}\|y(t)\|^2} \right)\right]\leq
\exp\left(d{\delta_{\epsilon}t}\right).
\]
\end{itemize}
\end{proposition}

\begin{proof}~
\begin{itemize}
\item[i)] The first part follows from inequality \eqref{zeroth}.
The second part follows exactly as  inequality \eqref{zeroth}. We left the details to the interested reader.
\item[ii)] 
We provide the proof for the first part.
The second part proceeds exactly as the first part and we left the details to the interested reader.

Let $\epsilon>0$ and $t\geq 0$ be fixed. Notice that
\[
\begin{split}
&x^{\epsilon}(t)-\varphi(t) =-\int\limits_{0}^{t}{\left[F(x^{\epsilon}(s))-F(\varphi(s))\right]\ud s}+\sqrt{\epsilon}B(t)\\
&\hspace{0.5cm}= -\int\limits_{0}^{t}{\left[\int\limits_{0}^{1}{DF{(\varphi(s)+\theta\left(x^{\epsilon}(s)-\varphi(s)\right))}\ud \theta}\right]
\left(x^{\epsilon}(s)-\varphi(s)\right)\ud s}+
\sqrt{\epsilon}B(t)\\
&\hspace{0.5cm}=-\int\limits_{0}^{t}{A^{\epsilon}(s)
\left(x^{\epsilon}(s)-\varphi(s)\right)\ud s}+\sqrt{\epsilon}B(t),
\end{split}
\]
where $A^{\epsilon}(s):=\int\limits_{0}^{1}{DF{(\varphi(s)+\theta\left(x^{\epsilon}(s)-\varphi(s)\right))}\ud \theta}$.
We will use the induction method.
The induction basis had already  proved in item i) of this proposition.
Consider $f_{n+1}(x)=\|x\|^{2(n+1)}$, $x\in \mathbb{R}^d$.
By the It\^o formula, it follows that
\begin{align*}
&\ud \|x^{\epsilon}(t)-\varphi(t)\|^{2(n+1)}=\\
&\hspace{1cm}\-2(n+1)\|x^{\epsilon}(t)-\varphi(t)\|^{2n}
\<x^{\epsilon}(t)-\varphi(t),A^{\epsilon}(t)\left(x^{\epsilon}(t)-\varphi(t)\right)\>\ud t\\
&\hspace{1cm} +\epsilon (d+2n)(n+1)\|x^{\epsilon}(t)-\varphi(t)\|^{2n}\ud t\\
&\hspace{1cm}+2(n+1)\sqrt{\epsilon}\|x^{\epsilon}(t)-\varphi(t)\|^{2n}
\<x^{\epsilon}(t)-\varphi(t),\ud B(t)\>.
\end{align*}
From \eqref{C2} we obtain
\begin{align*}
&\ud \|x^{\epsilon}(t)-\varphi(t)\|^{2(n+1)}\leq 
-2\delta (n+1)\|x^{\epsilon}(t)-\varphi(t)\|^{2(n+1)}\ud t\\
&\hspace{1cm}+\epsilon (d+2n)(n+1)\|x^{\epsilon}(t)-\varphi(t)\|^{2n}\ud t\\
&\hspace{1cm}+2(n+1)\sqrt{\epsilon}\|x^{\epsilon}(t)-\varphi(t)\|^{2n}
\<x^{\epsilon}(t)-\varphi(t),\ud B(t)\>.
\end{align*}
After a localisation argument, we can take expectation in both sides of the last differential inequality and deduce that
\begin{align*}
\tfrac{\ud}{\ud t} \bb E\left[\|x^{\epsilon}(t)-\varphi(t)\|^{2(n+1)}\right]\leq &
-2\delta (n+1)\bb E\left[\|x^{\epsilon}(t)-\varphi(t)\|^{2(n+1)}\right]\\
&+\epsilon (d+2n)(n+1)\bb E\left[\|x^{\epsilon}(t)-\varphi(t)\|^{2n}\right].
\end{align*}
By the induction hypothesis we have
\[
\bb E\left[\|x^{\epsilon}(t)-\varphi(t)\|^{2n}\right]\leq \frac{c_{n}\epsilon^{n}}{2^n\delta^n}\quad  \text{ for any } t \geq 0.
\]
Then
\begin{align*}
\tfrac{\ud}{\ud t} \bb E\left[\|x^{\epsilon}(t)-\varphi(t)\|^{2(n+1)}\right]\leq &
-2\delta (n+1)\bb E\left[\|x^{\epsilon}(t)-\varphi(t)\|^{2(n+1)}\right]\\
&+(n+1) \frac{c_{n+1}\epsilon^{n+1}}{2^n\delta^n}.
\end{align*}
From Lemma \ref{gin} we obtain
\[
\mathbb{E}\left[\|x^{\epsilon}(t)-\varphi(t)\|^{2(n+1)}\right]\leq
\frac{c_{n+1}\epsilon^{n+1}}{2^{n+1}\delta^{n+1}} \quad \text{ for any } t \geq 0.
\]
\item[ii)]
We provide the proof for the first part.
The second part follows exactly as the first part and again we left the details to the interested reader.

Let $\epsilon>0$ and $t\geq 0$ be fixed. By the Monotone Convergence Theorem it follows that
\[
\mathbb{E}\left[e^{r\frac{\left\|x^{\epsilon}(t)-\varphi(t)\right\|^2}{\epsilon}} \right]=
\sum\limits_{n=0}^{\infty}{\mathbb{E}\left[  {\frac{r^n\left\|x^{\epsilon}(t)-\varphi(t)\right\|^{2n}}{\epsilon^n n!}} \right]}.
\]

By item i) of this Proposition, we have
\[
\sum\limits_{n=0}^{\infty}{\mathbb{E}\left[  {\frac{r^n\left\|x^{\epsilon}(t)-\varphi(t)\right\|^{2n}}{\epsilon^n n!}} \right]}\leq
1+ \sum\limits_{n=1}^{\infty}{\frac{r^n c_n}{2^n \delta^n n!}}.
\]
Since  $\sum\limits_{n=1}^{\infty}{\frac{r^n c_n}{2^n \delta^n n!}}<+\infty$ when $0\leq r<\delta$, then we deduce the statement.
\item[iii)] 
We give the proof for the first part.
The second part proceeds exactly as the first part and again we left the details to the interested reader.

Let $\epsilon>0$ and $t\geq 0$ be fixed.
We will use the It\^o formula for the function $g_{\epsilon}(x)=e^{\kappa_\epsilon \|x\|^2}$, $x\in \mathbb{R}^d$,
 where $\kappa_{\epsilon}:=\frac{\delta_{\epsilon}}{\epsilon}$.
Then
\begin{align*}
\ud e^{\kappa_{\epsilon}{\|x^{\epsilon}(t)-\varphi(t)\|^2}{}}=&
-2\kappa_{\epsilon}e^{\kappa_{\epsilon}\|x^{\epsilon}(t)-\varphi(t)\|^2}
\<A^{\epsilon}(t)(x^{\epsilon}(t)-\varphi(t)),x^{\epsilon}(t)-\varphi(t)\>\ud t\\
&\hspace{-1.0cm}+\epsilon\left(2\kappa_{\epsilon}^2 e^{\kappa_{\epsilon}\|x^{\epsilon}(t)-\varphi(t)\|^2}\|x^{\epsilon}(t)-\varphi(t)\|^2+
\kappa_{\epsilon}d e^{\kappa_{\epsilon}\|x^{\epsilon}(t)-\varphi(t)\|^2}
\right)\ud t\\
&\hspace{-1.0cm}+2d\sqrt{\epsilon}\kappa_{\epsilon}e^{\kappa_{\epsilon}\|x^{\epsilon}(t)-\varphi(t)\|^2}
\< x^{\epsilon}(t)-\varphi(t),\ud B(t)\>.
\end{align*}
Using \eqref{C2} we obtain
\begin{align*}
\ud e^{\kappa_{\epsilon}{\|x^{\epsilon}(t)-\varphi(t)\|^2}{}}\leq &
-2\kappa_{\epsilon}\delta e^{\kappa_{\epsilon}\|x^{\epsilon}(t)-\varphi(t)\|^2}\|x^{\epsilon}(t)-\varphi(t)\|^2 \ud t\\
&\hspace{-1.0cm}+\epsilon\left(2\kappa_{\epsilon}^2 e^{\kappa_{\epsilon}\|x^{\epsilon}(t)-\varphi(t)\|^2}\|x^{\epsilon}(t)-\varphi(t)\|^2+
\kappa_{\epsilon}d e^{\kappa_{\epsilon}\|x^{\epsilon}(t)-\varphi(t)\|^2}
\right)\ud t\\
&\hspace{-1.0cm}+2d\sqrt{\epsilon}\kappa_{\epsilon}e^{\kappa_{\epsilon}\|x^{\epsilon}(t)-\varphi(t)\|^2}\< x^{\epsilon}(t)-\varphi(t),\ud B(t)\>.
\end{align*}
Since $0<\delta_\epsilon\leq \frac{\delta}{2}$ then
\begin{align*}\label{ttt}
\ud &e^{\kappa_{\epsilon}{\|x^{\epsilon}(t)-\varphi(t)\|^2}{}}\leq
-\kappa_{\epsilon}\delta e^{\kappa_{\epsilon}\|x^{\epsilon}(t)-\varphi(t)\|^2}\|x^{\epsilon}(t)-\varphi(t)\|^2 \ud t \\
&\hspace{1cm}+\epsilon
\kappa_{\epsilon}d e^{\kappa_{\epsilon}\|x^{\epsilon}(t)-\varphi(t)\|^2}
\ud t+
2d\sqrt{\epsilon}\kappa_{\epsilon}e^{\kappa_{\epsilon}\|x^{\epsilon}(t)-\varphi(t)\|^2}\< x^{\epsilon}(t)-\varphi(t),\ud B(t)\>.
\end{align*}
By item i) and  item ii) of this proposition and using a localisation argument  we deduce 
\[
\frac{\ud }{\ud t}\mathbb{E}\left[e^{\kappa_{\epsilon}{\|x^{\epsilon}(t)-\varphi(t)\|^2}{}}\right]\leq
\epsilon
\kappa_{\epsilon}d \mathbb{E}\left[e^{\kappa_{\epsilon}\|x^{\epsilon}(t)-\varphi(t)\|^2}\right]
\quad \textrm{ for any } t\geq 0.
\]
Now, using Lemma \ref{gin} we obtain
\[
\mathbb{E}\left[e^{\delta_{\epsilon}\frac{\left\|x^{\epsilon}(t)-\varphi(t)\right\|^2}{\epsilon}} \right]\leq e^{d\delta_{\epsilon}t}\quad \text{ for any } t \geq 0.
\]
\end{itemize}
\end{proof}

\begin{lemma}[Uniquely Ergodic]\label{ec4}
Assume that \eqref{C1} holds.
For any $\epsilon\in (0,1]$ there exists a unique invariant measure $\mu^{\epsilon}$ for the dynamics \eqref{SEDO}.
The unique probability invariant measure $\mu^{\epsilon}$ has exponential moments
\[
\int\limits_{\mathbb{R}^d}e^{\beta \|y\|}\mu^{\epsilon}(\ud y)<+\infty \quad \textrm{ for any } \beta\geq 0.
\]
In addition, for any $\beta>0$ there exist positive constants $C^{\epsilon,\beta}_1$ and $C^{\epsilon,\beta}_2 $ such that  for any initial condition $x_0\in \mathbb{R}^d$ we have
\[
\ud_{\mathrm{TV}}(x^{\epsilon}(t,x_0),\mu^{\epsilon})\leq 
C^{\epsilon,\beta}_1e^{-tC^{\epsilon,\beta}_2 }\left(
e^{\beta \|x_0\|}+\int\limits_{\mathbb{R}^d}e^{\beta \|y\|}\mu^{\epsilon}(\ud y) 
\right)
\]
for any $t\geq 0$. In particular,
\[
\lim\limits_{t\rightarrow +\infty}\ud_{\mathrm{TV}}(x^{\epsilon}(t,x_0),\mu^{\epsilon})=0.
\]
\end{lemma}
\begin{proof}
This follows immediately from Theorem $3.3.4$ page $91$ of \cite{KUL}. 
\end{proof}

\begin{lemma}\label{lne} 
Assume that \eqref{C1} holds.
Consider the matrix differential equation:
\begin{equation}
\label{EDO1}
\left\{
\begin{array}{r@{\;=\;}l}
\frac{\ud }{\ud t}\Sigma(t) & -DF(0)\Sigma(t)-\Sigma(t)(DF(0))^*+\mathrm{I}_d\quad \textrm{ for }t\geq 0,\\
\Sigma(0) & \Sigma_0,
\end{array}
\right.
\end{equation}
where $\Sigma_0$ is any $d\times d$  matrix. Then $\lim\limits_{t\rightarrow +\infty}\Big\|\Sigma(t)-\Sigma\Big\|=0$, where $\Sigma$ is the unique solution of the Lyapunov matrix equation:
\begin{equation}\label{lyapp}
DF(0)X+X(DF(0))^*=\mathrm{I}_d.
\end{equation}
\end{lemma}

\begin{proof}
Write $\Lambda=DF(0)$ and let $t\geq 0$. Notice that all eigenvalues of $\Lambda$ have positive real part. Denote by $\{\phi(t,x): t \geq 0\}$ the solution of the linear system:
\[
\left\{
\begin{array}{r@{\;=\;}l}
\tfrac{\ud}{\ud t} \phi(t) & -\Lambda \phi(t)\quad \textrm{ for } t\geq 0,\\
\phi(0) & x.
\end{array}
\right.
\]
Then $\{\phi(t,x): t \geq 0\}$ is globally asymptotic stable and consequently the Lyapunov matrix equation (\ref{lyapp}) has a unique positive definite solution $\Sigma$. From (\ref{lyapp}) it follows that $\Sigma$ is a symmetric matrix. Let
\[r(t):=\|\Sigma(t)-\Sigma\|^2=\sum\limits_{i,j=1}^{d}{(\Sigma_{i,j}(t)-\Sigma_{i,j})^2}\quad \textrm{ for any } t \geq 0.
\]
Let $\delta_{i,j}=1$ if $i=j$ and $\delta_{i,j}=0$ if $i\not=j$. Notice that
\[
\sum\limits_{k=1}^{d}{\Lambda_{i,k}\Sigma_{k,j}}+
\sum\limits_{k=1}^{d}{\Sigma_{i,k}\Lambda_{j,k}}=\delta_{i,j}\quad \textrm{ for any } i,j\in \{1,\ldots,d\}.
\]
Then
\begin{align*}
&\frac{\ud }{\ud t}r(t)=2\sum\limits_{i,j=1}^{d}{(\Sigma_{i,j}(t)-\Sigma_{i,j})\frac{\ud }{\ud t}\Sigma_{i,j}(t)}
=\\
&2\sum\limits_{i,j=1}^{d}{(\Sigma_{i,j}(t)-\Sigma_{i,j})\left(-\sum\limits_{k=1}^{d}{\Lambda_{i,k}(\Sigma_{k,j}(t)-\Sigma_{k,j})}-
\sum\limits_{k=1}^{d}{(\Sigma_{i,k}(t)-\Sigma_{i,k})\Lambda_{j,k}}\right)}.
\end{align*}
After rearrangement the sums and using \eqref{C1} we have
\begin{eqnarray*}
\frac{\ud }{\ud t}r(t) &\leq & -4\delta r(t)\quad \textrm{ for }t\geq 0,\\
r(0)&=& \|\Sigma_0-\Sigma\|^2.
\end{eqnarray*}
By Lemma \ref{gin} we deduce 
\[
\|\Sigma(t)-\Sigma\|^2\leq e^{-4\delta t}\|\Sigma_0-\Sigma\|^2 \quad \textrm{ for any } t\geq 0
\]
which implies the statement.
\end{proof}
\begin{remark}
Let $A$ be a $d$-squared matrix such that $F(x)=Ax$ satisfies \eqref{C1}.
If we take $F(x)=Ax$ in the stochastic differential equation (\ref{SEDO}), the covariance matrix associated to the solution of (\ref{SEDO}) satisfies the matrix differential equation (\ref{EDO1}) with initial datum $\Sigma_0$ the zero-matrix.
\end{remark}

\begin{lemma}\label{covmat}
Assume that \eqref{C2} holds.
The covariance matrix of $y(t)$ converge as $t\rightarrow +\infty$ to a non-degenerate covariance matrix $\Sigma$, where $\Sigma$ is the unique solution of the Lyapunov matrix equation:
\begin{equation*}
DF(0)X+X(DF(0))^*=\mathrm{I}_d.
\end{equation*}
\end{lemma}
\begin{proof}
For any $t\geq 0$, let $\Lambda(t)$ be the covariance matrix of the $y(t)$.
This matrix satisfies the matrix differential equation:
\begin{equation*}
\label{EDO5}
\left\{
\begin{array}{r@{\;=\;}l}
\frac{\ud }{\ud t}\Lambda(t) & -DF(\varphi(t))\Lambda(t)-\Lambda(t)(DF(\varphi(t)))^*+\mathrm{I}_d\quad \textrm{ for }t\geq 0,\\
\Sigma(0) & 0.
\end{array}
\right.
\end{equation*}

Let $\mathcal{K}_{x_0}:=\{x\in \mathbb{R}^d:\|x\|\leq \|x_0\|\}$. By  \eqref{C2} we have $\varphi(x,t)\in \mathcal{K}_{x_0}$ for any $x\in \mathcal{K}_{x_0}$ and $t\geq 0$. Since $F\in \mathcal{C}^2(\mathbb{R}^d, \mathbb{R}^d)$ there exists a constant $L:=L_{x_0}>0$
such that $\|DF(x)-DF(0)\|\leq L\|x\|$ for any $x\in \mathcal{K}_{x_0}$.

Take $\eta>0$ and $\tau_{\eta}:=\frac{1}{\delta}\ln\left(\frac{\|x_0\|}{\eta}\right)$ such that
\begin{equation}\label{lipt}
\|DF(\varphi(t))-DF(0)\|\leq L\|\varphi(t)\|\leq L\|x_0\|e^{-\delta t}\leq L\eta
\end{equation}
for every $t\geq \tau_{\eta}$. Call $\tau:=\tau_{\eta}$. Then,
\begin{equation*}
\left\{
\begin{array}{r@{\;=\;}l}
\frac{\ud }{\ud t}\Delta(t) & -DF\left(0\right)\Delta(t)-\Delta(t)(DF\left(0\right))^*+ \mathrm{I}_d \quad \textrm{ for } t\geq 0,\\
\Delta(0) & \Lambda(\tau).
\end{array}
\right.
\end{equation*}

Let $\Pi(t)=\Lambda(t+\tau)-\Delta(t)$, $t\geq 0$. Then
\begin{equation*}
\label{EDO6}
\left\{
\begin{array}{r@{\;=\;}l}
\frac{\ud }{\ud t}\Lambda(t) & -DF(\varphi(t+\tau))\Pi(t)-\Pi(t)(DF(\varphi(t+\tau)))^*+g(t,\tau)\quad \textrm{ for }t\geq 0,\\
\Pi(0) & 0,
\end{array}
\right.
\end{equation*}
where $g(t,\tau):=(DF(0)-DF(\varphi(t+\tau)))\Delta(t)+\Delta(t)(DF(0)-DF(\varphi(t+\tau)))^*$ for  $t\geq 0$.
Therefore
\begin{align*}
&\frac{\ud }{\ud t}\|\Pi(t)\|^2=2\sum\limits_{i,j=1}^{d}{\Pi_{i,j}(t)\frac{\ud }{\ud t}\Pi_{i,j}(t)}
=\\
&2\sum\limits_{i,j=1}^{d}{\Pi_{i,j}(t)
\left(-\sum\limits_{k=1}^{d}{DF(\varphi(t+\tau))_{i,k}\Pi_{k,j}(t)}-
\sum\limits_{k=1}^{d}{\Pi_{i,k}(t)DF(\varphi(t+\tau))_{j,k}}+R_{i,j}(t)\right)},
\end{align*}
where
\begin{align*}
R_{i,j}(t)=&\sum\limits_{k=1}^{d}{(DF(0)_{i,k}-DF(\varphi(t+\tau))_{i,k})\Delta_{k,j}(t)}\\
&+\sum\limits_{k=1}^{d}{\Delta_{i,k}(t)(DF(0)-DF(\varphi(t+\tau))_{j,k})^*}.
\end{align*}
From \eqref{C2} we deduce 
\begin{align*}
\frac{\ud }{\ud t}\|\Pi(t)\|^2\leq -4\delta\|\Pi(t)\|^2+\sum\limits_{i,j=1}^{d}{\Big|\Pi_{i,j}(t)R_{i,j}(t)\Big |}
\quad \textrm{ for any } t\geq 0.
\end{align*}
Moreover, using Lipschitz condition (\ref{lipt}) and Lemma \ref{lne} we obtain
\[
\sum\limits_{i,j=1}^{d}\Big|\Pi_{i,j}(t)R_{i,j}(t)\Big |\leq C\eta+ C\eta \|\Pi(t)\|^2\quad \textrm{ for any } t\geq 0,
\]
where $C$ is a positive constant. {\em A priori} we can take $0<\eta<\frac{3\delta}{C}$  and using Lemma \ref{gin} we obtain
\[
\|\Pi(t)\|^2\leq \frac{C\eta}{\delta}{(1-e^{-\delta t})}<\frac{C\eta}{\delta}\quad \textrm{ for any } t\geq 0.
\]
Letting $t\rightarrow +\infty$ and after $\eta\rightarrow 0$
we deduce  that $\lim\limits_{t\rightarrow +\infty}\|\Pi(t)\|^2=0$ which together with Lemma \ref{lne} imply
$\lim\limits_{t\rightarrow +\infty}\Lambda(t)=\Sigma$.
\end{proof}

\begin{lemma}[Gronwall Inequality]\label{gin}
Let $T>0$ be fixed. Let $g:[0,T]\rightarrow \mathbb{R}$ be a  $\mathcal{C}^1$-function and $h:[0,T]\rightarrow \mathbb{R}$ be a  $\mathcal{C}^0$-function.
If
\[
\frac{\ud}{\ud t}g(t)\leq -ag(t)+h(t)\quad \textrm{ for any } t\in [0,T],
\] 
where $a\in \mathbb{R}$, and the derivative at $0$ and $T$ are understanding as the right and left derivatives, respectively.  Then 
\[
g(t)\leq e^{-at}g(0)+e^{-at}\int\limits_{0}^{t}{e^{as}h(s)}\ud s \quad \textrm{ for any }
t\in[0,T].
\]
Moreover,
\[
|g(t)|\leq e^{-at}|g(0)|+\frac{(1-e^{-at})}{a}\max\limits_{s\in[0,t]}|h(s)|
\quad \textrm{ for any }
t\in[0,T].
\]
\end{lemma}

\section*{Acknowledgments}
G. Barrera would like to thank to CONACyT-M\'exico for the post-doctorate grant $2016$-$2017$ at the Department of Probability and Statistics, CIMAT. He acknowledges financial support from CNPq-Brazil and the Instituto de Matem\'atica Pura e Aplicada, IMPA for the grant for supporting his Ph.D. studies and the Summer postdoctoral stay $2017$, where part of this work was done. Also, he would like to express his gratitude to FORDECyT-CONACyT-M\'exico for the travel support to the Summer postdoctoral stay $2017$ at IMPA.

\markboth{}{References}
\bibliographystyle{amsplain}

\end{document}